\documentclass[letterpaper,12pt]{article}

\usepackage{amsmath,amsthm}
\usepackage{amsfonts}
\usepackage{latexsym}
\usepackage{amssymb,bm}
\usepackage{mathrsfs}
\usepackage{mathtools}
\usepackage[page]{appendix}
\usepackage{float}
\usepackage{placeins}
\usepackage{rotating}
\usepackage{tabulary} 
\usepackage{changes}

\usepackage{tikz}
\usepackage{subfig}
\usepackage{graphicx}
\usepackage{pgfplots}
\usepackage{xcolor}
\usetikzlibrary{calc,trees,positioning,arrows,chains,shapes.geometric,decorations.pathreplacing,decorations.pathmorphing,shapes,matrix,shapes.symbols}
\usetikzlibrary{shapes,backgrounds}
\pgfplotsset{width=13cm,compat=1.14}
\usepackage[all,knot,arc,import,poly]{xy}
\usepackage{array}
\usepackage{booktabs}

 \usepackage{grffile}

\usepackage{color}

\usepackage{newfloat}
\usepackage{caption}
\usepackage{placeins}
\DeclareFloatingEnvironment[fileext=frm,placement={ht},name=Frame]{algfloat}

\usetikzlibrary{arrows.meta}

\newtheorem{theorem}{Theorem}[section]
\newtheorem{lemma}[theorem]{Lemma}

\newtheorem{proposition}[theorem]{Proposition}

\newtheorem{definition}[theorem]{Definition}

\newtheorem{algorithm}{Algorithm}

\newtheorem{assumption}{Assumption}

\DeclareMathOperator*{\argmin}{argmin}

\newcommand{\inner}[2]{\langle #1,#2\rangle}

\newcommand{\norm}[1]{\|{#1}\|}

\newcommand{\R}{\mathbb{R}}
\newcommand{\tos}{\rightrightarrows}

\newcommand{\comenta}[1]{}

\newcommand{\N}{\mathbb{N}}

\newcommand{\mgap}{\vspace{.1in}}

\pgfplotstableread[row sep=\\,col sep=&]{
    simname     &   relerr    &   primDR    &  PrimDR\\
    Ball64      &   0.280e3   &   0.278e3    &   0.123e3     \\
    Logo64      &   0.283e3   &   0.282e3    &   0.139e3     \\
    Mug32       &   0.153e3   &   0.153e3    &   0.136e3    \\
    Mug128      &   0.920e3   &   0.914e3    &   0.435e3     \\
    Finance     &   0.974e3   &   1.709e3    &   1.079e3    \\
    PEMS        &   3.354e3   &   3.648e3    &   1.088e3    \\
    Brain       &   1.855e3   &   2.295e3    &   1.219e3     \\
    Colon       &   0.450e3   &   0.482e3    &   0.256e3       \\
   Leukemia    &   0.675e3    &   0.774e3    &   0.424e3    \\
   Lymphoma    &   0.908e3    &   0.925e3    &   482      \\
   Prostate    &   1.520e3    &   1.739e3    &   998      \\
   Srbct       &   0.426e3    &   0.401e3    &   221      \\
}\dispdata

\pgfplotstableread[row sep=\\,col sep=&]{
    name     &   rel  &   prim  &  Prim\\
    Ball64      &   603     &   382     &   191   \\
    Logo64      &   621    &   369    &   212   \\
    Mug32       &   998    &   307     &   302  \\
    Mug128      &   1214    &   1046    &   488  \\
   Colon       &   5847     &   1401    &   1461   \\
   Leukemia    &   7888    &   2321    &   2543      \\
   Lymphoma    &   15266     &   3179     &   3083     \\
   Prostate    &   20612    &   5193    &   6629     \\
   Srbct       &   6213    &   1505     &   1334     \\
}\dispdata

\pgfplotstableread[row sep=\\,col sep=&]{
   A           &   B        &   C  &  D\\
   Finance     &   18944    &  7852    &   9737    \\
   PEMS        &   85858   &   9318   &   9235   \\
   Brain       &   24612   &   7116   &   7655  \\
   }\dispdata
   
\pgfplotstableread[row sep=\\,col sep=&]{
    Name     &   err  &   DR  &  Relx\\
    Ball64      &   11.02    &   7.86     &   3.75   \\
    Logo64      &   11.37   &   7.62  &   4.04  \\
    Mug32       &   1.07    &   0.51     &   0.43  \\
   Brain       &   13.59  &   5.94   &   5.53 \\
   Colon       &   1.56    &   0.45    &   0.28   \\
   Leukemia    &   4.24   &   2.23    &   1.59     \\
   Lymphoma    &   7.18     &   2.63    &   2.03     \\
   Prostate    &   33.21    &   13.15    &   11.88     \\
   Srbct       &   1.83   &   0.42  &   0.35     \\
}\dispdata

\pgfplotstableread[row sep=\\,col sep=&]{
      E           & F           &         G   & H \\
      Mug128      &   248.38    &   218.08    &   101.17  \\
      Finance     &   805.17    &  327.56    &   347.97    \\
      PEMS        &   7546.11  &   1092.16   &   988.12  \\
  }\dispdata
  
  \pgfplotstableread[row sep=\\,col sep=&]{
          I       &   J    &   K  & L & M \\
    Colon     &  2666    &   2145   &  1979 & 1578   \\
    Leukemia  & 1662  & 1116  & 922 & 788   \\
    Prostate  &   1936 & 1583  & 1677 & 1198 \\
    Arcene    & 419  &  276  & 359 & 290\\
  }\dispdata
  
  \pgfplotstableread[row sep=\\,col sep=&]{
          N       &   O    &   P  & Q & R & S\\
    Colon     &  20612 & 23919 & 21697  & 26247 & 8283   \\
    Leukemia  &  7715 & 12086  & 11625 & 6536 & 4448 \\
    Prostate  &  18901 & 27505 & 24548 & 13730 & 10997 \\
    Arcene    &  780 & 3236 & 3589 & 4648 & 1450 \\
  }\dispdata
  
  \pgfplotstableread[row sep=\\,col sep=&]{
          T       &   U    &   V  & W & X & Y \\
    Colon     & 182.36  & 36.52 & 91.57  & 73.29 & 12.82  \\
    Leukemia  & 112.74 & 105.42 & 241.14 & 60.95 & 23.05 \\
    Prostate  & 342.16 & 719.67 & 850.81 & 206.39  & 128.69 \\
    Arcene    & 122.72 & 312.11 & 370.94 & 184.35 & 46.13 \\
  }\dispdata

\begin{document}

\title{Relative-error inertial-relaxed inexact versions of Douglas-Rachford and ADMM splitting algorithms}
\author{
   M. Marques Alves
\thanks{
Departamento de Matem\'atica,
Universidade Federal de Santa Catarina,
Florian\'opolis, Brazil, 88040-900 ({\tt maicon.alves@ufsc.br}).
The work of this author was partially supported by CNPq grants no.
405214/2016-2 and 304692/2017-4.}
\and
Jonathan Eckstein
\thanks{Department of Managment Science and Information Systems and RUTCOR,
        Rutgers Business School Newark and New Brunswick,
        Piscataway, NJ 08854, USA ({\tt jeckstei@business.rutgers.edu}).  The work
        of this author was partially supported by National Science Foundation grant 
        CCF-161761 and Air Force Office of Scientific Research grant
        FA9550-15-1-0251.}
\and
  Marina Geremia
\thanks{
Departamento de Matem\'atica,
Universidade Federal de Santa Catarina,
Florian\'opolis, Brazil, 88040-900 ({\tt marinageremia@yahoo.com.br}).}
\and
Jefferson G. Melo 
\thanks{
IME, Universidade Federal de Goi\'as, Goi\^ania, Brazil, 74001-970 ({\tt jefferson.ufg@gmail.com}).}
}
\date{May 9, 2014}
\date{}


\maketitle

\begin{abstract}
This paper derives new inexact variants of the Douglas-Rachford splitting
method for maximal monotone operators and the alternating direction method of
multipliers (ADMM) for convex optimization.  The analysis is based on a new
inexact version of the proximal point algorithm that includes both an inertial
step and overrelaxation.  We apply our new inexact ADMM method to LASSO
and logistic regression problems and obtain somewhat better computational
performance than earlier inexact ADMM methods.
\end{abstract}

\pagestyle{plain}


\section{Introduction}
\label{sec:intro}
This paper develops a sequence of three algorithms, each building on the
previous one.  The first algorithm is a new variant of the proximal point
algorithm~\cite{roc-mon.sjco76} for the general, abstract problem $0\in T(z)$,
where $T$ is a set-valued maximal monotone operator on $\R^n$ for which
$T^{-1}(0)\neq \emptyset$. Our proposed method is a new inertial variant of
the relaxed hybrid proximal projection (HPP) method introduced
in~\cite{Sol-Sv:hy.unif}; see also~\cite{sol.sva-hyb.jca99}. It lacks the full
generality of~\cite{Sol-Sv:hy.unif}, but introduces a new ``inertial'' step
modification.

Using this first algorithm, we then develop a new inexact variant of the
Douglas-Rachford (DR) splitting method for monotone inclusion problems of the of
form $0 \in A(x) + B(x)$, where $A,B: \R^n \tos \R^n$ are set-valued maximal
monotone operators.

Finally, based on this latter method, we derive a new inexact variant of the
alternating direction method of multipliers (ADMM) algorithm for solving
convex optimization problems of the form $\min_{x\in\R^n}\{f(x) + g(x)\}$,
where $f,g: \R^n \to \R
\cup \{+\infty\}$ are closed proper convex functions.  Using the well known
LASSO and logistic regression problems as examples,
we perform some computational tests on this last algorithm 
in Section~\ref{sec:exp} below, finding somewhat better practical
performance than earlier proposed inexact ADMM methods
from~\cite{eck.yao-app.coap17,eck.yao-rel.mp17}.

This path for developing approximate DR and ADMM methods was pioneered
in~\cite{eck.ber-dou.mp92}, and is also taken in the more recent paper
by Eckstein and Yao~\cite{eck.yao-rel.mp17}: in each case,
one takes an approximate form of proximal point algorithm
(PPA)~\cite{roc-mon.sjco76} and uses it to obtain an approximate form of DR
splitting, which can then be used to obtain a new ``primDR''
variant of the ADMM; the iteration complexity of the ``primDR'' ADMM was later
studied in \cite{alv.ger-dr.numa18}.  The main difference between this paper
and the development of ``primDR'' in~\cite{eck.yao-rel.mp17} is in the
underlying variant of the PPA.  The ``primDR'' analysis used the hybrid
proximal extragradient (HPE) method~\cite{sol.sva-hyb.svva99} due to Solodov
and Svaiter, whereas here we instead use the new inexact HPP developed in
Section~\ref{sec:alg}.

Our general approach resembles that
of~\cite{eck.yao-rel.mp17} in that it uses a primal derivation and the
``coupling matrix'' between $f$ and $g$ in the optimization
formulation must be the identity,
whereas~\cite{eck.ber-dou.mp92}, drawing on early work 
in~\cite{gabay1983applications}, uses a dual derivation and allows for more
general coupling matrices.  Our analysis is also much closer
to~\cite{eck.yao-rel.mp17} than that of~\cite{eck.yao-app.coap17}, which uses
a primal-dual ``Lagrangian splitting'' analysis patterned
after~\cite{fortin1983augmented}. 


Inertial algorithms for convex optimization and monotone
inclusions~\cite{alv.att-iner.svva01} have been a subject of
intense research in recent years. They appear in connection 
with continuous dynamics --- see,
\emph{e.g.}~\cite{alv.att-iner.svva01,
att.chb.pey-fas.mp18,att.pey.red-fas.jde16}
--- accelerated first- and second-order algorithms, and operator splitting
    methods --- see \emph{e.g.}~\cite{att.cab-con.pre218,att.cab-con.pre18,
    che.cha.ma-ine.sjis15,che.ma.yan-gen.sijo15,lor.poc-ine.jmiv15} --- with
    good theoretical and practical performance
    improvements over prior methods. The inertial methods we
    propose here have the novel property of simultaneously combining inexact
    iterations, inertia, and relaxation, with
    the maximum inertial step $\alpha$ and maximum relaxation factor
    $\bar\rho$ being subject to a mutual constraint; see~\eqref{eq:alpha}
    and~\eqref{eq:betatau} below.  However, the inertial and relaxation
    parameters may be chosen independently of the relative-error tolerances.


The remainder of this paper is organized as follows:
Section~\ref{sec:alg} presents our inertial-relaxed HPP method (Algorithm~\ref{inertial.hpp}) and its convergence analysis (Theorems~\ref{th:main} and~\ref{th:wc}).
Section~\ref{sec:dr} then uses the HPP method to develop an
inexact inertial-relaxed DR method (Algorithm~\ref{alg:ine_dr}), for which
convergence is established in
Theorem~\ref{th:conv_dr}.
Section~\ref{sec:admm_iner} then uses inertial-relaxed DR
method to derive a partially inexact relative-error ADMM method
(Algorithm~\ref{admm_inertial}). The main result of this section is
Theorem~\ref{th:conv_admm}.
Section \ref{sec:exp} presents numerical experiments on LASSO and logistic
regression problems.


\section{An inertial-relaxed hybrid proximal projection \\ (HPP) method}
 \label{sec:alg}
We begin by developing a new method for the problem
\begin{align}
  \label{eq:mip}
 0\in T(z),
\end{align}
where $T:\R^n\tos\R^n$ is a maximal monotone operator; we assume that this
problem has a solution.
Our new proposed procedure for this problem, related to the method
of~\cite{Sol-Sv:hy.unif} but having a new ``inertial'' step feature, is given
below as Algorithm~\ref{inertial.hpp}.
%

\begin{algfloat}
\noindent
\fbox{
\addtolength{\linewidth}{-2\fboxsep}%
\addtolength{\linewidth}{-2\fboxrule}%
\begin{minipage}{0.98\linewidth}
\begin{algorithm}
\label{inertial.hpp}
{\bf A relative-error inertial-relaxed HPP method for solving \bf{(\ref{eq:mip})}}
\end{algorithm}
\begin{itemize}
\item[] {\bf Initialization:} Choose  $z^0=z^{-1}\in \R^n$ and $0\leq \alpha, \sigma< 1$ and $0<\underline{\rho}<\overline{\rho}<2$
\item [ ] {\bf for} $k=0,1,\dots$  {\bf do}
\item [ ] Choose $\alpha_{k}\in [0,\alpha]$ and define
  \begin{align}
      \label{eq:ext.hpe}
     w^{k} = z^{k}+\alpha_{k}(z^{k}-z^{k-1})
 \end{align}
\item [ ] Find $(\tilde z^k,v^k)\in \R^n\times \R^n$ and $\lambda_k>0$ such that
\begin{align}
\label{eq:err.hpe}
 v^k\in T(\tilde z^k),\quad \norm{\lambda_k v^k+\tilde z^k-w^{k}}^2 \leq \sigma^2 \left(\norm{\tilde z^k-w^{k}}^2+\norm{\lambda_k v^k}^2\right)
\end{align}
\item[ ] If $v^k=0$, then {\bf stop}. Otherwise, choose $\rho_k\in [\,\underline{\rho},\overline{\rho}\,]$ and set
 \begin{align}
  \label{eq:err.hpe2}
   z^{k+1}=w^{k}-\rho_k \dfrac{\inner{w^k-\tilde z^k}{v^k}}{\norm{v^k}^2}v^k
 \end{align}
\item[ ] {\bf end for}  
\end{itemize}
\noindent
\end{minipage}
} 
\end{algfloat}

We make the following remarks concerning this algorithm:
\begin{itemize}
\item[(i)]  The extrapolation step in~\eqref{eq:ext.hpe} introduces inertial
effects --- see \emph{e.g.}~\cite{alv-wea.siam03,alv.att-iner.svva01} ---
controlled by the parameter $\alpha_k$.  The effect of the overrelaxation
parameter $\rho_k$ in~\eqref{eq:err.hpe2} is similar but not identical, as
shown in Figure~\ref{fig:arrows} below. Conditions on $\{\alpha_k\}$,
$\alpha\in [0,1)$ and $\overline{\rho}\in (0,2)$ that guarantee the
convergence of Algorithm \ref{inertial.hpp} are given in Theorem \ref{th:wc}
--- see \eqref{eq:alpha} and \eqref{eq:betatau} and Figure \ref{fig02} below.
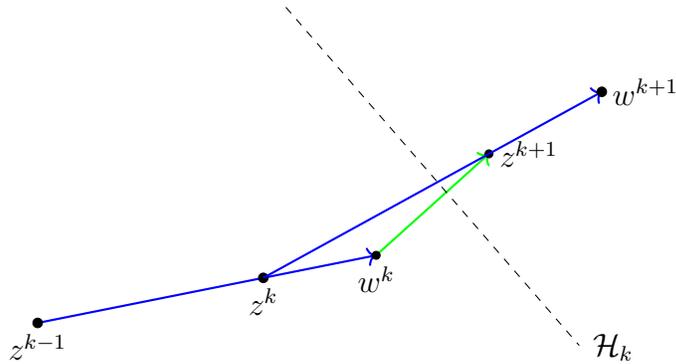
\begin{figure}[!htb]
\centering
\vspace{0.5cm}
\begin{tikzpicture}[scale=3]
\draw[thick,->,blue] (0, 0) -- (1.5, 0.3);
\draw[thick,->,green] (1.5, 0.3) -- (2, 0.75);
\draw[dashed] (1.1,1.4) -- (2.40,-0.1);
\filldraw (0,0) circle (0.6pt) node[align=center,   below] {$z^{k-1}$};
\filldraw (1,0.2) circle (0.6pt) node[align=center, below] {$z^{k}$};
\filldraw (1.5,0.3) circle (0.5pt) node[align=center, below] {$w^{k}$};
\filldraw (2,0.75) circle (0.5pt) node[align=center, right] {$z^{k+1}$};
\draw (2.55,0) circle (0.0pt) node[align=center, below] {$\mathcal{H}_k$};
\draw[thick,->,blue] (1,0.2) -- (2.5,1.025);
\filldraw (2.5,1.025) circle (0.6pt) node[align=center, right] {$w^{k+1}$};
\end{tikzpicture}
\vspace{-0.5ex}
\caption{Geometric interpretation of steps~\eqref{eq:ext.hpe}
and~\eqref{eq:err.hpe2} in Algorithm~\ref{inertial.hpp}.  The overrelaxed
projection step~\eqref{eq:err.hpe2} is orthogonal to the separating hyperplane
$\mathcal{H}_k$, which can differ from the direction between $z^{k-1}$, $z^k$,
and $w^k$ when $\alpha_k > 0$.}\label{fig:arrows}
\end{figure}
\item[(ii)] If $\alpha=0$, in which case $\alpha_k\equiv 0$,
Algorithm~\ref{inertial.hpp} reduces to a special case of the HPP method
of~\cite{Sol-Sv:hy.unif}; see also \cite{sol.sva-hyb.jca99}.
Algorithm~\ref{inertial.hpp} is also closely related to the inertial version
of the HPP method presented in~\cite{alv-wea.siam03}, although that method
uses a different relative error criterion.
\item[(iii)] At each iteration $k$, 
condition~\eqref{eq:err.hpe} is a relative error criterion for
the inexact solution of the proximal subproblem
$\tilde{z}^k = (I + \lambda_k T)^{-1}(w^k) := J_{\lambda_k T}(w^k)$. If
$\sigma=0$, then this equation must be solved exactly and the pair $(\tilde
z^k, v^k)$ may be written $(\tilde z^k, v^k)=(J_{\lambda_k T}(w^k),
\lambda_k^{-1}(w^k-\tilde z^k))$.
Here, we  are primarily concerned with situations in which the calculation
of $J_{\lambda_k T}(w^k)$ is relatively difficult and must be approached with
an iterative algorithm.  In such cases, we use the
condition~\eqref{eq:err.hpe} as an acceptance criterion to truncate such an
iterative calculation, possibly saving computational effort.  We do not
specify the exact form of the iterative algorithm used to produce a pair
$(\tilde{z}^k,v^k)$ satisfying~\eqref{eq:err.hpe}, as it depends on the class
of problems to which the algorithm is being applied (and thus the structure of
the operator $T$).  See~\cite{sol.sva-hyb.jca99,Sol-Sv:hy.unif} for a related
discussion; an abstract formalism of the class of algorithm needed to find a
solution to~\eqref{eq:err.hpe} is the ``$\mathcal{B}$-procedure'' described
in~\cite{eck.yao-rel.mp17} and also used in Section~\ref{sec:dr} below.

%
\item[(iv)] The point $z^{k+1}$ in~\eqref{eq:err.hpe2} may be viewed 
as $z^{k+1}=w^k+\rho_k(P_{\mathcal{H}_k}(w^k)-w^k)$, where $P_{\mathcal{H}_k}$
denotes orthogonal projection onto the hyperplane 
\begin{equation}\label{hyperplane}
\mathcal{H}_k:=\{z\in \R^n\;|\;\inner{z}{v^k}=\inner{\tilde z^k}{v^k}\},
\end{equation}
which strictly separates
$w^k$ from the solution set $T^{-1}(0)$ of \eqref{eq:mip}.  This kind of
projective approach to approximate proximal point algorithms was pioneered
in~\cite{sol.sva-hyb.jca99}.

\item[(v)]  Algorithm~\ref{inertial.hpp} is an inexact variant of the
proximal point algorithm (PPA)~\cite{roc-mon.sjco76}.  In particular, each of
its iterations performs an approximate resolvent calculation subject a
relative error criterion, and then executes a projection operation in the
manner introduced in~\cite{sol.sva-hyb.jca99};
see~\cite{sol.sva-hyb.svva99,Sol-Sv:hy.unif} for related work.  The main
difference from~\cite{sol.sva-hyb.jca99} is the inertial
step~\eqref{eq:ext.hpe}. 
\end{itemize}

If $v^k=0$ in Algorithm \ref{inertial.hpp}, then it follows from the inclusion
in~\eqref{eq:err.hpe}  that $\tilde z^k$ is a solution of~\eqref{eq:mip},
that is, $0\in T(\tilde z^k)$, so we halt immediately with the
solution $\tilde{z}^k$.  For the remainder of this section, we assume that
$v^k \not\equiv 0$ and hence that Algorithm~\ref{inertial.hpp} generates an
infinite sequence of iterates.
%
%
The following well-known identity will be useful in the analysis of
Algorithm~\ref{inertial.hpp}:
\begin{align}
  \label{eq:hard}
 \norm{(1-\rho)p+\rho
 q}^2=(1-\rho)\norm{p}^2+\rho\norm{q}^2-\rho(1-\rho)\norm{p-q}^2\quad \forall
 p,q\in \R^n\quad \forall \rho\in \R.
\end{align}

\begin{lemma}\emph{\cite[Lemma 2]{Sol-Sv:hy.unif}}
\label{lm:gauss}
For each $k \geq 0$, condition~\eqref{eq:err.hpe} implies that
\begin{align}
 \label{eq:gauss}
  \dfrac{1-\sigma^2}{1+\sqrt{1-(1-\sigma^2)^2}}\norm{\tilde z^k - w^k}\leq \norm{\lambda_k  v^k}\leq \dfrac{1-\sigma^2}{1-\sqrt{1-(1-\sigma^2)^2}}\norm{\tilde z^k - w^k}.
\end{align}
\end{lemma}
\noindent An immediate implication of Lemma~\ref{lm:gauss} is that $v^k=0$ if 
and only if $\tilde z^k=w^k$.


The proof of the following proposition can be found, using different notation, in \cite{Sol-Sv:hy.unif}. For the convenience of the
reader, we also present it here.

\begin{proposition}
\label{inq:err2}
Let $\{z^k\}$, $\{\tilde z^k\}$ and $\{w^k\}$ be generated 
by \emph{Algorithm \ref{inertial.hpp}}
and define, for all $k\geq 1$,
\begin{align}
  \label{eq:born}
 s_{k} = (2- \overline\rho)\max\left\{\overline{\rho}^{-1} \norm{z^{k}-w^{k-1}}^2,\;\; \underline{\rho}\, (1-\sigma^2)^{\,2}
\norm{\tilde z^{k-1}-w^{k-1}}^2\right\}.
\end{align}
Then, for any  $z^*\in T^{-1}(0)$,
\begin{align}
 \label{eq:103}
 \norm{z^{k+1}-z^*}^2 + s_{k+1} \leq \norm{w^k-z^*}^2, \qquad \forall k\geq 0.
%
%
\end{align}
\end{proposition}
\begin{proof}
We start by defining $\widehat z^{\,k+1}$ as the orthogonal projection of
$w^k$ onto the hyperplane $\mathcal{H}:=\{z\in
\R^n\,|\,\inner{z}{v^k}=\inner{\tilde z^k}{v^k}\}$, \emph{i.e.},
\begin{align}
  \label{eq:wolfe}
  \widehat z^{\,k+1}:=w^k-\dfrac{\inner{w^k-\tilde z^k}{v^k}}{\norm{v^k}^2}v^k.
\end{align}
Next we show that the hyperplane $\mathcal{H}$ stricly separates the current
point $w^k$ from the solution set $\Omega:=T^{-1}(0)\neq \emptyset$, that is,
\begin{align}
 \label{eq:frank}
 \inner{w^k}{v^k}>\inner{\tilde z^k}{v^k}\geq \inner{z^*}{v^k}\quad \forall z^*\in \Omega.
\end{align}
To this end, $0\in T(z^*)$, $v^k\in T(\tilde
z^k)$ and the monotonicity of $T$ yield $\inner{\tilde z^k-z^*}{v^k}\geq 0$,
which is equivalent to the second inequality in \eqref{eq:frank}.
On the other hand, note that from \eqref{eq:err.hpe} and the Young inequality 
$2ab\leq a^2+b^2$ we have
\begin{align*}
 \inner{w^k-\tilde z^k}{v^k}&\geq \dfrac{1-\sigma^2}{2 \lambda_k}\left(\norm{\tilde z^k-w^k}^2+\norm{\lambda_k v^k}^2\right)
     \geq (1-\sigma^2)\norm{w^k-\tilde z^k}\norm{v^k},
\end{align*}
which in turn yields
\begin{align}
 \label{eq:fermi}
   \dfrac{\inner{w^k-\tilde z^k}{v^k}}{\norm{v^k}}
   \geq (1-\sigma^2)\norm{w^k-\tilde z^k}>0.
\end{align}
 One consequence of
\eqref{eq:fermi} is the first inequality in \eqref{eq:frank},
so~\eqref{eq:frank} must hold.

From~\eqref{eq:wolfe} and~\eqref{eq:frank}, we may infer that
$\widehat{z}^{k+1}$ is the projection $w^k$ onto the halfspace $\{z\in
\R^n\;|\;\inner{z}{v^k}\leq\inner{\tilde z^k}{v^k}\}$, which is a convex set
containing $z^*$.  The well-known firm nonexpansivess properties of the
projection operation then imply that
\begin{align}
 \label{eq:riccati}
 \norm{w^k-z^*}^2-\norm{\widehat z^{\,k+1}-z^*}^2&\geq 
\norm{w^k-\widehat z^{\,k+1}}^2.
\end{align}
Algebraic manipulation of \eqref{eq:err.hpe2} and \eqref{eq:wolfe}
yields $z^{\,k+1} - z^*=(1-\rho_k)(w^k - z^*)+\rho_k (\widehat
z^{\,k+1}-z^*)$.  Combining this equation with \eqref{eq:hard} with
$(p,q)=(w^k-z^*,\widehat z^{\,k+1}-z^*)$ gives
\begin{align*}
 \norm{z^{k+1} - z^*}^2&=(1-\rho_k)\norm{w^k-z^*}^2+\rho_k\norm{\widehat z^{\,k+1}-z^*}^2-\rho_k(1-\rho_k)\norm{w^k-\widehat z^{\,k+1}}^2,
\end{align*}
%
which after some rearrangement yields
%
\begin{align*}
\norm{w^k-z^*}^2-\norm{z^{k+1}-z^*}^2 
  &= \rho_k\left(\norm{w^k-z^*}^2-\norm{\widehat z^{\,k+1}-z^*}^2\right)
     + \rho_k(1-\rho_k)\norm{w^k-\widehat z^{\,k+1}}^2.
\end{align*}
Using~\eqref{eq:riccati} in the first term on the right-hand side of this
identity produces
%
\begin{align}
\norm{w^k-z^*}^2-\norm{z^{k+1}-z^*}^2
   &\geq \rho_k \norm{w^k-\widehat z^{\,k+1}}^2
      + \rho_k(1-\rho_k)\norm{w^k-\widehat z^{\,k+1}}^2 \nonumber \\
   &= (\rho_k+\rho_k(1-\rho_k))\norm{w^k-\widehat z^{\,k+1}}^2 \nonumber \\
 \label{eq:villani}
 &=\rho_k(2-\rho_k)\left(\dfrac{\inner{w^k-\tilde z^k}{v^k}}{\norm{v^k}}\right)^2
   \quad\quad\;\;\, \text{[by~\eqref{eq:wolfe}]}\\
 \label{eq:kant}
&\geq \rho_k(2-\rho_k)(1-\sigma^2)^2\norm{w^k-\tilde z^k}^2.
\quad\quad\text{[by~\eqref{eq:fermi}]}
\end{align}
To finish the proof, we observe that \eqref{eq:villani} and 
\eqref{eq:err.hpe2} yield
\begin{align*}
 \norm{w^k-z^*}^2-\norm{z^{k+1}-z^*}^2&\geq \rho_k^{-1}(2-\rho_k)\norm{z^{k+1}-w^k}^2.
\end{align*}
Combining this inequality with \eqref{eq:kant}, \eqref{eq:born} and the bounds
$\rho_k\in [\,\underline \rho, \overline \rho\,]$ results in \eqref{eq:103}.
\end{proof}

The inequality~\eqref{eq:ineq.tech} presented in the following proposition
plays a role in the convergence analysis of inertial proximal algorithms ---
see \emph{e.g.}~\cite{alv.att-iner.svva01} --- similar to that played by
Fej\'er monotonicity in the analysis of standard proximal algorithms.

\begin{proposition}
 \label{lm:tech}
 Let $\{z^k\}$, $\{w^k\}$ and $\{\alpha_k\}$ be generated by \emph{Algorithm
\ref{inertial.hpp}} and let $\{s_k\}$ be as in \eqref{eq:born}. Further let
$z^*\in T^{-1}(0)$ and define
 \begin{align}
 \label{eq:def.phi}
  (\forall k\geq -1)\quad \varphi_k:=\norm{z^k-z^*}^2\;\;\mbox{and}\;\;\;
 (\forall k\geq 0)\quad
\delta_k:=\alpha_{k}(1+\alpha_{k})\norm{z^{k}-z^{k-1}}^2.
\end{align}
Then, $\varphi_0=\varphi_{-1}$ and
 \begin{align}
  \label{eq:ineq.tech}
 \varphi_{k+1}-\varphi_{k}+s_{k+1} \leq
   \alpha_{k}(\varphi_{k}-\varphi_{k-1})+\delta_{k}\qquad \forall k\geq 0,
\end{align}
that is, the sequences $\{\varphi_k\}$, $\{s_k\}$, $\{\alpha_k\}$ and
$\{\delta_k\}$ satisfy the assumptions of Lemma \ref{lm:alv.att} below.
\end{proposition}
\begin{proof}
From~\eqref{eq:ext.hpe} we obtain $z^{k}-z^*=(1+\alpha_{k})^{-1}(w^{k}-z^*)+
 \alpha_{k}(1+\alpha_{k})^{-1}(z^{k-1}-z^*)$,
which in conjunction with~\eqref{eq:hard} and some algebraic manipulation yields
\begin{align*}
 \norm{w^{k}-z^*}^2=(1+\alpha_{k})\norm{z^{k}-z^*}^2-\alpha_{k}\norm{z^{k-1}-z^*}^2
+\alpha_{k}(1+\alpha_{k})\norm{z^{k}-z^{k-1}}^2.
\end{align*}
Using the above identity and \eqref{eq:def.phi} we
obtain, for all $k\geq 0$, that
\begin{align*}
 \norm{w^{k}-z^*}^2=(1+\alpha_{k})\varphi_{k}-\alpha_{k}\varphi_{k-1}+\delta_{k}.
\end{align*}
From~\eqref{eq:103} in Proposition~\ref{inq:err2} and the definition of
$\varphi_k$ in~\eqref{eq:def.phi}, the above inequality
yields~\eqref{eq:ineq.tech}. Finally, $\varphi_0=\varphi_{-1}$ follows
from the initialization $z^0=z^{-1}$ and the first definition in
\eqref{eq:def.phi}.
\end{proof}

The following theorem presents our first result on the asymptotic convergence
of Algorithm~\ref{inertial.hpp} under the summability assumption
\eqref{eq:th:main.01}. Next, Theorem~\ref{th:wc} gives sufficient
conditions~\eqref{eq:alpha} and~\eqref{eq:betatau} on the inertial and
relaxation parameters to assure that
\eqref{eq:th:main.01} is satisfied.

\begin{theorem}[Convergence of Algorithm \ref{inertial.hpp}]
 \label{th:main}
 Let $\{z^k\}$, $\{\tilde z^k\}$, $\{v^k\}$ $\{\lambda_k\}$ and $\{\alpha_k\}$ be generated by \emph{Algorithm \ref{inertial.hpp}}.
%
If $\inf_k \lambda_k>0$ and
\begin{align}
 \label{eq:th:main.01}
 \sum_{k=0}^\infty\,\alpha_{k}\norm{z^{k}-z^{k-1}}^2<+\infty
\end{align}
then $\{z^k\}$ converges to a solution of  the monotone inclusion problem \eqref{eq:mip}. Moreover, $\{\tilde z^k\}$ converges to the same solution and $\{v^k\}$ converges to zero.
\end{theorem}
\begin{proof}
Define $\{s_k\}$ is as in~\eqref{eq:born}.
Using Proposition \ref{lm:tech}, \eqref{eq:th:main.01}, that $\alpha_k\leq
\alpha<1$ for all $k\geq 0$, and Lemma \ref{lm:alv.att}, it follows that (i)
$\lim_{k\to \infty}\,\|z^k-z^*\|$ exist for every $z^*\in
\Omega:=T^{-1}(0)\neq \emptyset$ and $\sum_{k=1}^\infty\,s_k<+\infty$. So, in
particular, $\{z^k\}$ is bounded and (ii) $\lim_{k\to \infty}\,s_k=0$ .
From the form of~\eqref{eq:born}, that $\lim_{k\to \infty}\,s_k=0$, and the  assumption that $\inf \lambda_k>0$, and Lemma \ref{lm:gauss}, we conclude that
\begin{align}
  \label{eq:ter01}
  \lim_{k\to \infty}\,\|z^k-w^{k-1}\|=
 \lim_{k\to \infty}\,\|\tilde z^k-w^{k}\|=\lim_{k\to \infty}\,\|v^k\|=0.
\end{align}
Now let $z^\infty\in \R^n$ be any cluster point of the bounded sequence
$\{z^k\}$. By~\eqref{eq:ter01}, this point is also a cluster point of
$\{w^k\}$ and $\{\tilde{z}^k\}$. Let $\{k_j\}_{j=0}^\infty$ be an increasing
sequence of indices such that $\tilde z^{k_j}\to z^{\infty}$.  We then have
\begin{align*}
 (\forall j\geq 0)\;\;\; v^{k_j}\in T (\tilde z^{k_j}),\;\; \lim_{j\to \infty}\,v^{k_j}=0 \;\;\mbox{and}\;\;  \lim_{j\to \infty}\,\tilde z^{k_j}=z^{\infty},
\end{align*}
which by the standard closure property of maximal monotone operators yields
$z^\infty\in \Omega=T^{-1}(0)$. Hence, the desired result on $\{z^k\}$ follows
from (i) and Opial's lemma (stated below as Lemma~\ref{lm:opial}). On the other hand, the convergence of
$\{z^k\}$ and
\eqref{eq:ter01} yields the remaining results regarding $\{\tilde z^k\}$ and
$\{v^k\}$.
\end{proof}

 \begin{figure}
    \centering
        \begin{tikzpicture}[scale=2] \centering
    \draw[->,line width = 0.50mm] (-0.2,0) -- (1.4,0) node[right] { $\beta$};
    \draw[->,line width = 0.50mm] (0,-0.2) -- (0,2.4) node[above] {$\bar\rho(\beta)$};
    \draw[domain=0:1,smooth,variable=\x,red,line width = 0.50mm,scale=1] plot ({\x},{(2*(\x -1)^2)/((2*(\x -1)^2)+3*\x -1)});
    \draw[red] (0,2) circle (0.35mm);
    \draw[red] (1,0) circle (0.35mm);
    \draw[dashed] (0.3333,0) -- (0.3333,1);
     \draw[dashed] (0,1) -- (0.3333,1);
        \node[below left,black] at (0,0) {0};
        \node[below right ,black] at (1,0) {1};
        \node[above left,black] at (0,2) {2};
        \node[below ,black] at (0.3333,0) {$\frac{1}{3}$};
        \node[left,black] at (0,1) {1};
    \end{tikzpicture}
    \caption{{The relaxation parameter upper bound
    $\overline\rho(\beta)$ from \eqref{eq:betatau} 
    as a function of inertial step upper bound $\beta>0$
    of \eqref{eq:alpha}. Note that $\overline \rho(1/3)=1$, while
    $\overline \rho(\beta)>1$ whenever $\beta<1/3$.}}
    \label{fig02}
    \end{figure}
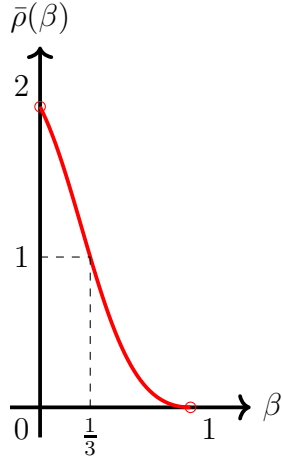

\begin{theorem}[Convergence of Algorithm \ref{inertial.hpp}]
 \label{th:wc}
Let $\{z^k\}$, $\{\alpha_k\}$ and $\{\lambda_k\}$ be generated by \emph{Algorithm \ref{inertial.hpp}}.
Assume that $\alpha\in [0,1)$, $\overline{\rho}\in (0,2)$ and $\{\alpha_k\}$ satisfy the following (for some $\beta>0$):
\begin{align}
 \label{eq:alpha}
 0\leq \alpha_{k}\leq \alpha_{k+1}\leq \alpha<\beta<1\qquad \forall k\geq 0
\end{align}
and
\begin{align}
\label{eq:betatau}
\overline{\rho}=\overline{\rho}(\beta):=\dfrac{2(\beta-1)^2}{2(\beta-1)^2+3\beta-1}.
\end{align}
Then, 
\begin{align}
  \label{eq:sum}
 \sum_{k=1}^\infty\norm{z^k-z^{k-1}}^2< +\infty.
\end{align}
As a consequence, it follows that under the assumptions \eqref{eq:alpha} and \eqref{eq:betatau} the sequence $\{z^k\}$
generated by \emph{Algorithm \ref{inertial.hpp}} converges to a solution of the monotone inclusion
problem \eqref{eq:mip} whenever $\inf \lambda_k>0$.  Moreover, under the above assumptions, $\{\tilde z^k\}$ converges to the same solution and $\{v^k\}$ converges to zero.
\end{theorem}
\begin{proof}
Using \eqref{eq:ext.hpe}, the Cauchy-Schwarz inequality and the Young inequality
$2ab\leq a^2+b^2$ with $a:=\norm{z^{k+1}-z^{k}}$ and $b:=\norm{z^{k}-z^{k-1}}$
we find
\begin{align}
 \norm{z^{k+1}-w^{k}}^2
   \nonumber
    &=\norm{z^{k+1}-z^{k}}^2+\alpha_{k}^2\norm{z^{k}-z^{k-1}}^2-2\alpha_{k}\inner{z^{k+1}-z^{k}}{z^{k}-z^{k-1}}\\
  \nonumber
    & \geq \norm{z^{k+1}-z^{k}}^2+\alpha_{k}^2\norm{z^{k}-z^{k-1}}^2-\alpha_{k} \left(2\norm{z^{k+1}-z^{k}}\norm{z^{k}-z^{k-1}}\right)\\
    & \geq (1-\alpha_{k}) \norm{z^{k+1}-z^{k}}^2-\alpha_{k}(1-\alpha_{k})\norm{z^{k}-z^{k-1}}^2. \label{startthm2.5}
\end{align}
Starting with a rearrangement of~\eqref{eq:ineq.tech}, we then obtain
\begin{align}
 &\varphi_{k+1}-\varphi_{k}-
   \alpha_{k}(\varphi_{k}-\varphi_{k-1}) \leq \delta_{k}- s_{k+1}\nonumber\\
  &\quad\leq \alpha_{k}(1+\alpha_{k})\norm{z^{k}-z^{k-1}}^2-(2- \overline\rho)\overline{\rho}^{-1} \norm{z^{k+1}-w^{k}}^2
  \qquad\qquad\qquad\qquad \text{[by \eqref{eq:born} and \eqref{eq:def.phi}]}\nonumber\\
  &\quad\leq \alpha_{k}(1+\alpha_{k})\norm{z^{k}-z^{k-1}}^2 \nonumber\\
  &\qquad\qquad-(2- \overline\rho)\overline{\rho}^{-1} 
        \left[(1-\alpha_{k}) \norm{z^{k+1}-z^{k}}^2
        -\alpha_{k}(1-\alpha_{k})\norm{z^{k}-z^{k-1}}^2\right]
        \qquad\text{[by \eqref{startthm2.5}]}\nonumber\\
  &\quad=-(2- \overline\rho)\overline{\rho}^{-1}(1-\alpha_{k}) \norm{z^{k+1}-z^{k}}^2+ \left[\alpha_{k}(1+\alpha_{k})+(2- \overline\rho)\overline{\rho}^{-1}\alpha_{k}(1-\alpha_{k})\right]\norm{z^{k}-z^{k-1}}^2, \nonumber\\
  &\quad=-(2- \overline\rho)\overline{\rho}^{-1}(1-\alpha_{k}) \norm{z^{k+1}-z^{k}}^2+ 
  \gamma_k\norm{z^{k}-z^{k-1}}^2, \label{almostdone}
\end{align}
where
\begin{align}
 \label{eq:304}
  \gamma_{k}:= -2(\overline{\rho}^{\,-1}-1)\alpha_k^2+2\,\overline{\rho}^{\,-1}\alpha_k\qquad \forall k\geq 0.
\end{align}
Some elementary algebraic manipulations of~\eqref{almostdone} then yield
\begin{align}
 \label{eq:301}
 \varphi_{k+1}-\varphi_{k}-\alpha_{k}(\varphi_{k}-\varphi_{k-1})-\gamma_{k} \norm{z^{k}-z^{k-1}}^2
\leq -(2\,\overline{\rho}^{\,-1}-1) (1-\alpha_{k})\norm{z^{k+1}- z^{k}}^2\quad \forall k\geq 0.
\end{align}
Define now the scalar function:
\begin{align}
  \label{eq:def.q}
%
q(\nu):= 2(\overline{\rho}^{\,-1}-1)\nu^{\,2}-(4\overline{\rho}^{\,-1}-1)\nu + 2\overline{\rho}^{\,-1}-1, 
\end{align}
and
\begin{align}
 \label{eq:305}
\mu_0:=(1-\alpha_0)\varphi_0\geq 0,\qquad
\mu_k:=\varphi_k-\alpha_{k-1}\varphi_{k-1}+\gamma_{k}\norm{z^k-z^{k-1}}^2\quad \forall k\geq 1,
\end{align}
where $\varphi_k$ is as in \eqref{eq:def.phi}.
 Using \eqref{eq:301}-\eqref{eq:305} and the assumption that $\{\alpha_k\}$ is nondecreasing --- see \eqref{eq:alpha} --- we obtain, for all $k\geq 0$,
\begin{align}
 \label{eq:307}
 \nonumber
 \mu_{k+1}-\mu_{k}
\nonumber
  &\leq
\left[\varphi_{k+1}-\varphi_{k}-\alpha_{k}(\varphi_{k}-\varphi_{k-1})-
\gamma_{k}\norm{z^{k}-z^{k-1}}^2\right]+\gamma_{k+1}\norm{z^{k+1}-z^{k}}^2\\
 \nonumber
 &\leq \left[\gamma_{k+1}- (2\overline{\rho}^{\,-1}-1) (1-\alpha_{k+1})\right]\norm{z^{k+1}-z^{k}}^2\\
\nonumber
&=-\left[2(\overline{\rho}^{\,-1}-1)\alpha_{k+1}^2-(4\overline{\rho}^{\,-1}-1)\alpha_{k+1}+2\overline{\rho}^{\,-1}-1\right]\norm{z^{k+1}-z^{k}}^2\\
&=-q(\alpha_{k+1})\norm{z^{k+1}-z^{k}}^2.
\end{align}
We will now show that $q(\alpha_{k+1})$ admits a uniform positive lower bound.
To this end, note first that from \eqref{eq:betatau} and Lemma
\ref{lm:inverse} below that we have
\begin{align*}
 \beta =\dfrac{2(2-\overline{\rho})}{4-\overline{\rho}+\sqrt{\overline{\rho} (16-7\overline{\rho})}}.
\end{align*}
Using the latter identity, \eqref{eq:def.q}, and Lemma \ref{lm:quadratic}
below with $a=2(\overline{\rho}^{\,-1}-1)$, $b=4\overline{\rho}^{\,-1}-1$, and
$c=2\overline{\rho}^{\,-1}-1$, we conclude that $q(\cdot)$ is decreasing in
$[0,\beta]$ and $\beta>0$ is a root of $q(\cdot)$.
Thus, in view of~\eqref{eq:alpha}, we conclude that
\begin{align}
  \label{eq:noite04}
 q(\alpha_{k+1})\geq q(\alpha)>q(\beta)=0,
\end{align}
which gives the desired uniform positive lower bound on $q(\alpha_{k+1})$.

Using \eqref{eq:307} and \eqref{eq:noite04} we find
\begin{align}
  \label{eq:ineq.mu}
 \norm{z^{k+1}-z^{k}}^2\leq \dfrac{1}{q(\alpha)}(\mu_{k}-\mu_{k+1}),\quad \forall k\geq 0,
\end{align}
which, in turn, combined with \eqref{eq:alpha} and the definition of $\mu_k$ in \eqref{eq:305},  gives
\begin{align}
  \label{eq:sum.q}
 \nonumber
 \sum_{j=0}^k\,\norm{z^{j+1}-z^{j}}^2&\leq \dfrac{1}{q(\alpha)}(\mu_0-\mu_{k+1}),\\
             &\leq \dfrac{1}{q(\alpha)}(\mu_0+\alpha \varphi_{k}) \quad \forall k\geq 0.
\end{align}
Note now that \eqref{eq:ineq.mu}, \eqref{eq:alpha} and \eqref{eq:305} also yield
\begin{align*}
  \mu_0\geq \ldots \geq \mu_{k+1}=&\varphi_{k+1}-\alpha_{k}\varphi_{k}+\gamma_{k+1}\|z^{k+1}-z^{k}\|^{2} \\
    \geq&  \varphi_{k+1}-\alpha\varphi_{k},\quad \forall k\geq 0,
\end{align*}
and so,
\begin{equation}
 \label{H13}
    \varphi_{k+1}\leq \alpha^{k+1}\varphi_{0}+\frac{\mu_{0}}{1-\alpha}\leq \varphi_{0}+\frac{\mu_{0}}{1-\alpha} \qquad \forall k\geq -1.
\end{equation}
Hence, \eqref{eq:sum} follows directly from \eqref{eq:sum.q} and \eqref{H13}.
On the other hand, the second statement of the theorem follows from \eqref{eq:sum} and Theorem \ref{th:main} (recall that
$\alpha_k\leq \alpha<1$ for all $k\geq 0$).
\end{proof}

\noindent We close this section with a few further remarks about the analysis
of Algorithm~\ref{inertial.hpp}:
\begin{itemize}
\item[(i)] Conditions \eqref{eq:alpha} and
\eqref{eq:betatau} on $\{\alpha_k\}$, $\alpha$ and $\overline{\rho}$ 
guarantee that the summability condition \eqref{eq:th:main.01} is satisfied,
thus guaranteeing the convergence of Algorithm~\ref{inertial.hpp}.  
Similar conditions were also recently proposed and studied 
in~\cite{alv.mar-ine.svva19,att.cab-con.pre18}.
Since
Algorithm~\ref{inertial.hpp} is be the basis of the DR and ADMM methods
developed in the next two sections, conditions
\eqref{eq:alpha} and \eqref{eq:betatau} will also play an important role in
their convergence analyses.
\item[(ii)] If we set $\beta=1/3$ in \eqref{eq:alpha}, then it
follows immediately
from \eqref{eq:betatau} that $\overline{\rho}=1$. On the
other hand, we have $\overline{\rho}>1$ in \eqref{eq:betatau} whenever
$\beta<1/3$  (see also Figure~\ref{fig02}). Setting $\beta=1/3$ in
\eqref{eq:alpha} is corresponds to the standard strategy in the literature of
inertial proximal algorithms; see \emph{e.g.} \cite{alv.att-iner.svva01,
che.cha.ma-ine.sjis15}.
\end{itemize}


\section{A partially inexact inertial-relaxed Douglas-Rachford (DR) algorithm}
\label{sec:dr}

Consider the monotone inclusion problem of finding $z\in \R^n$ such that
\begin{align}
  \label{eq:mip_ab}
 0\in A(z)+B(z)
\end{align}
where $A$ and $B$ are (set-valued) maximal monotone operators on $\R^n$ for which
the solution set $(A+B)^{-1}(0)$ of \eqref{eq:mip_ab} is nonempty.

A popular operator splitting algorithms for finding approximate
solutions to \eqref{eq:mip_ab} is the Douglas-Rachford (DR) algorithm
\cite{dou.rac-num.tams56, lio.mer-spl.sjna79, eck.ber-dou.mp92}:
\begin{align}
 \label{eq:dr_iter}
  z^{k+1} = J_{\gamma A}\big(2J_{\gamma B}(z^k)-z^k\big)+z^k-J_{\gamma B}(z^k)
  \quad \forall k\geq 0,
\end{align}
where $\gamma>0$ is a scaling parameter, $z^k$ is the current iterate and
$J_{\gamma A} = (\gamma A+I)^{-1}$ and $J_{\gamma B} = (\gamma B+I)^{-1}$ are
the resolvent operators of $A$ and $B$, respectively. The DR algorithm
\eqref{eq:dr_iter} is a splitting algorithm for solving the (structured)
inclusion \eqref{eq:mip_ab} in the sense that the resolvents $J_{\gamma A}$
and $J_{\gamma B}$ are employed separately, but the resolvent
$J_{\gamma(A+B)}$ of $A+B$ is not.  Such methods may be useful in
situation in which the values of $J_{\gamma A}$ and $J_{\gamma B}$ are
relatively easy to evaluate in comparison to those of $J_{\gamma(A+B)}$.

This section will develop an inexact version of the DR algorithm
\eqref{eq:dr_iter} for the situation in which the resolvent of one of the
operators, say $B$, is relatively hard, but evaluating $J_{\gamma A}$ is a
simple calculation. To this end, we consider the following equivalent
formulation of \eqref{eq:dr_iter} (see, \emph{e.g.}, \cite{eck.ber-dou.mp92}):
given some $r^k,b^k\in \R^n$,
\begin{align}
\label{eq:dr.spp} 
&\mbox{Find}\;\; (s^{k+1},b^{k+1})\in B\;\;\mbox{such that}\;\;
s^{k+1}+\gamma b^{k+1} = r^k+\gamma b^k;\\ 
\label{eq:dr.spp02} 
 &\mbox{Find}\;\; (r^{k+1},a^{k+1})\in A\;\;\mbox{such that}\;\; 
r^{k+1}+\gamma a^{k+1} = s^{k+1} - \gamma b^{k+1}.
\end{align}
In this case, $z^k=r^k+\gamma b^k$. Since the resolvent $J_{\gamma A}$ of $A$
is assumed to be easily computable, the pair $(r^{k+1},a^{k+1})$ in
\eqref{eq:dr.spp02} is explicitly given by
\begin{align*}
r^{k+1} &= J_{\gamma A}(s^{k+1}-\gamma b^{k+1}) & 
a^{k+1} &= \gamma^{-1}(s^{k+1}-r^{k+1}) - b^{k+1}.
\end{align*}
For $B$, we by contrast suppose that exact computation of the pair $(s^{k+1},
b^{k+1})$ satisfying \eqref{eq:dr.spp} requires a relatively time-consuming
iterative process, which we model immediately below by the notion of a
$\mathcal{B}$-procedure as introduced in~\cite{eck.yao-rel.mp17}.
We first remark that~\eqref{eq:dr.spp} can be posed in the
more general framework of solving monotone inclusion problems of the form
\begin{align}
 \label{eq:jordan10}
 0\in s + \gamma B(s)  - (r+\gamma b),
\end{align}
where $r,b\in \R^n$ and $\gamma>0$.

\mgap

\begin{definition} [$\mathcal{B}$--procedure for solving \eqref{eq:jordan10}]
 \label{def:bp}
 A $\mathcal{B}$--procedure for (approximately) solving any instance of
\eqref{eq:jordan10} is  a mapping $\mathcal{B}:\R^n\times \R^n\times
\R_{++}\times \R^n\times \R^n\times \N^*\to \R^n\times \R^n$ such that if one
lets $(s^\ell, b^\ell)=\mathcal{B}(r, b,\gamma,\bar s, \bar b, \ell)$ for all
$\ell \in \N^*$ and any given $ r,b,\bar s, \bar b\in \R^n$ and $\gamma>0$,
then $b^\ell\in B(s^\ell)$, for all $\ell\in \N^*$,  the sequence $\{(s^\ell,
b^\ell)\}$  is convergent, and $s^\ell+\gamma b^\ell \to r+\gamma b$.
\end{definition}

\noindent
Following~\cite{eck.yao-rel.mp17}, the intuitive meaning of $(s^\ell,
b^\ell)=\mathcal{B}(r, b,\gamma,\bar s, \bar b, \ell)$ is that $(s^\ell,
b^\ell)$ is the $\ell^{\text{th}}$ trial approximation generated by some
iterative procedure for solving~\eqref{eq:jordan10}, starting from some
initial guess $(\bar s, \bar b)\in \R^n\times \R^n$. We refer the interested
reader to \cite[Section 5]{eck.yao-rel.mp17} for a more detalied discussion
and interpretation on the $\mathcal{B}$-procedure concept.

We make the following standing assumption:

\begin{assumption} 
 \label{asp:bp}
There exists a $\mathcal{B}$-procedure (according to Definition \ref{def:bp})
for approximately solving any instance of \eqref{eq:jordan10}.
\end{assumption}

\begin{algfloat}
\noindent
\fbox{
\addtolength{\linewidth}{-2\fboxsep}
\addtolength{\linewidth}{-2\fboxrule}
\begin{minipage}{0.96\linewidth}
\begin{algorithm}
\label{alg:ine_dr}
{\bf A partially inexact inertial--relaxed Douglas-Rachford splitting algorithm for solving \eqref{eq:mip_ab}}
\end{algorithm}
\begin{itemize}
\item[] Choose $\gamma>0$, $0\leq \alpha, \sigma<1$ and $0<\underline{\rho}<\overline{\rho}<2$. 
           Initialize $(s^0,b^0,r^0)=(s^{-1},b^{-1},r^{-1})\in (\R^n)^3$. 
\item [] {\bf for}\, $k=0,1,2,\dots$\,{\bf do}
\begin{enumerate}
\item[] Choose $\alpha_k\in [0,\alpha]$ and define 
  \begin{align}
    \label{eq:cartano}
   (\hat s^k,\hat b^k,\hat r^k)=(s^k,b^k,r^k)+\alpha_k\left [(s^k,b^k,r^k)-(s^{k-1},b^{k-1},r^{k-1})\right]
   \end{align}
\end{enumerate}
\item[] {\bf repeat}\,$\lbrace \mbox{{\bf for}}\,\ell=1,2,\dots\rbrace$
\begin{enumerate}
\item[] Improve the solution to 
\begin{align}
 \label{eq:jordan}
  b^{k,\ell}\in B(s^{k,\ell}),\quad s^{k,\ell}+\gamma b^{k,\ell}\approx \hat r^k+\gamma \hat b^k
\end{align}
%
by setting 
\begin{align}
  \label{eq:ferreiro}
 (s^{k,\ell},b^{k,\ell})=\mathcal{B}(\hat r^k, \hat b^k, \gamma, \hat s^k, \hat b^k,\ell)
\end{align}
%
(thus incrementally executing a step of the
$\mathcal{B}$--procedure)
\item[] Exactly find $(r^{k,\ell}, a^{k,\ell})\in \R^n\times \R^n$ such that
\begin{align}
 \label{eq:tartaglia}
    a^{k,\ell}\in A(r^{k,\ell}),\quad r^{k,\ell}+\gamma a^{k,\ell}=s^{k,\ell}-\gamma b^{k,\ell}
\end{align}
\end{enumerate}

\item[] {\bf until}  
\begin{align}
 \label{eq:lagrange}
  \norm{s^{k,\ell}+\gamma b^{k,\ell}-(\hat r^k+\gamma \hat b^k)}^2\leq \sigma^2\left(\norm{r^{k,\ell}+\gamma b^{k,\ell}-(\hat r^k+\gamma \hat b^k)}^2+\norm{s^{k,\ell}-r^{k,\ell}}^2\right)
\end{align}
\item[] {\bf if} 
$s^{k,\ell}=r^{k,\ell}$, then {\bf stop}
\item[] {\bf otherwise}, choose $\rho_k\in
[\,\underline{\rho},\overline{\rho}\,]$ and set
\begin{align}
 \label{eq:abel}
 & s^{k+1}=s^{k,\ell},\quad r^{k+1}=r^{k,\ell}\\
 \label{eq:cayley}
 & \theta_{k+1}=\dfrac{\inner{(\hat r^k - r^{k,\ell})+\gamma(\hat b^k - b^{k,\ell})}{s^{k,\ell}-r^{k,\ell}}}{\norm{s^{k,\ell}-r^{k,\ell}}^2}\\
 \label{eq:galois}
 & b^{k+1}=\hat b^k - \gamma^{-1}\left[(1-\rho_k\,\theta_{k+1})r^{k+1}+\rho_k\,\theta_{k+1} s^{k+1}-\hat r^k\right]
\end{align}
\item[] {\bf end for}
  \end{itemize}
\end{minipage}
} 
\end{algfloat}

We now combine the hypothesized $\mathcal{B}$-procedure with
an acceptance criterion for the
approximate solution of~\eqref{eq:dr.spp}.  We will follow the general
approach of~\cite{eck.yao-rel.mp17}, which is to exploit the connection
between the DR algorithm \eqref{eq:dr.spp}-\eqref{eq:dr.spp02} and the
proximal point algorithm as established in~\cite{eck.ber-dou.mp92}.
Specifically,
%
%
the DR algorithm \eqref{eq:dr.spp}, \eqref{eq:dr.spp02} is a special instance of the PP algorithm  in the sense that,
\begin{align*}
 r^{k+1}+\gamma b^{k+1} = (S_{\gamma, A, B}+I)^{-1}(r^k+\gamma b^k)\qquad \forall k\geq 0
\end{align*}
where the ``splitting" operator $S_{\gamma,A,B}$ is defined as~\cite{eck.ber-dou.mp92}
\begin{align}
 \label{eq:def_S}
 S_{\gamma,A,B}:=\{(r+\gamma b, s-r)\in \R^n\times \R^n\,|\, b\in B(s),\;\; a\in A(r),\;\; \gamma a+r=s-\gamma b\}.
\end{align}
The operator defined in \eqref{eq:def_S} is maximal monotone and 
\begin{align}
 \label{eq:eck_bert}
  (A+B)^{-1}(0) = J_{\gamma B}\big(S_{\gamma, A, B}^{\,-1}(0)\big),
\end{align}
which, in particular, gives that any solution $z^*\in \R^n$ of the monotone inclusion problem \eqref{eq:mip} with $T:=S_{\gamma,A,B}$, namely
\begin{align}
  \label{eq:mip_so}
 0\in S_{\gamma, A, B}(z)
\end{align}
yields a solution $x^*:=J_{\gamma B}(z^*)$ of \eqref{eq:mip_ab}.

Here, we follow a similar derivation
to~\cite{eck.yao-rel.mp17}, but use
Algorithm~\ref{inertial.hpp} of Section~\ref{sec:alg} to~\eqref{eq:mip_so} in
place of the HPE method of~\cite{sol.sva-hyb.svva99}. The result is an
inertial-relaxed inexact relative-error DR algorithm for solving
\eqref{eq:mip_ab}. We should emphasize that even $\alpha_k \equiv 0$ (there is
no inertial step) and $\rho_k \equiv 1$ (no overrelaxation), the resulting
algorithm differs from that of \cite{eck.yao-rel.mp17}.  This difference
arises because the underlying ``convergence engine'' of
Algorithm~\ref{inertial.hpp} is a form of hybrid proximal-projection (HPP)
algorithm, whereas~\cite{eck.yao-rel.mp17} used an HPE algorithm in the
equivalent role, using an extragradient step instead of
projection.

The proposed algorithm for solving \eqref{eq:mip_ab} is shown as
Algorithm~\ref{alg:ine_dr}.   We should mention that a different
inexact DR splitting algorithm in which relative errors are allowed in both
\eqref{eq:jordan} and \eqref{eq:tartaglia} was recently proposed and studied
in \cite{sva-ful.preprint18}, but without computational testing.
The following proposition shows that Algorithm \ref{alg:ine_dr} is indeed a
special instance of Algorithm \ref{inertial.hpp} for solving \eqref{eq:mip}
with $T:=S_{\gamma, A, B}$.

\begin{proposition}
 \label{pr:2e1}
 Consider the sequences evolved by Algorithm~\ref{alg:ine_dr} and for each
 $k\geq 0$ let $\ell(k)$ denote the value of $\ell$ for
 which~\eqref{eq:lagrange} is satisfied. For each $k \geq 0$, define, with $\gamma$ as in Algorithm~\ref{alg:ine_dr}, 
\begin{align}
 \begin{aligned}
 \label{eq:cauchy}
 z^k &:=r^k+\gamma b^k &  w^k &:=\hat r^k+\gamma \hat b^k\\
 \tilde z^k&:=r^{k,\ell(k)}+\gamma b^{k,\ell(k)} & v^k &:=s^{k,\ell(k)}-r^{k,\ell(k)}.
\end{aligned}
\end{align}
Then these latter sequences satisfy the conditions
\eqref{eq:ext.hpe}-\eqref{eq:err.hpe2} of \emph{Algorithm \ref{inertial.hpp}}
with $\lambda_k  \equiv 1$ and $T=S_{\gamma, A,B}$.
\end{proposition}
\FloatBarrier  
\begin{proof}
Fix any $k \geq 0$.
From \eqref{eq:cartano} and the definitions of $z^k$ and $w^k$
in \eqref{eq:cauchy} we have
\begin{align*}
 w^k& =\hat r^k+\gamma \hat b^k= r^k+\gamma b^k +\alpha_k [r^k+\gamma b^k - (r^{k-1}+\gamma b^{k-1})]\\
   & = z^k + \alpha_k(z^k-z^{k-1}),
\end{align*}
which is exactly \eqref{eq:ext.hpe}. Now note that the inclusion in
\eqref{eq:err.hpe} follows from the fact that $T:=S_{\gamma,A,B}$,
\eqref{eq:def_S}, \eqref{eq:tartaglia}, $b^{k,\ell(k)} \in B(s^{k,\ell(k)})$
from~\eqref{eq:jordan}, and the definitions of $v^k$ and $\tilde z^k$ in
\eqref{eq:cauchy}.

Further, \eqref{eq:cauchy} and \eqref{eq:lagrange} yield
\begin{align*}
 \norm{v^k+\tilde z^k-w^k}^2
 &= \norm{s^{k,\ell(k)}+\gamma b^{k,\ell(k)}-(\hat r^k+\gamma \hat b^k)}^2\\
 &\leq \sigma^2 \left(\norm{r^{k,\ell(k)}+\gamma b^{k,\ell(k)}
      -(\hat r^k+\gamma \hat b^k)}^2+\norm{s^{k,\ell}-r^{k,\ell}}^2\right)\\
  & = \sigma^2 \left(\norm{\tilde z^k-w^k}^2 + \norm{v^k}^2\right),
\end{align*}
which is exactly the inequality in  \eqref{eq:err.hpe} with $\lambda_k=1$.
%
Finally,
\begin{align*}
  z^{k+1} & = r^{k+1} + \gamma b^{k+1} && [\text{by~\eqref{eq:cauchy}}]\\
    & = r^{k,\ell(k)}
      +\gamma \left( \hat b^k-\dfrac{1}{\gamma}
         \left[(1-\rho_k\theta_{k+1})r^{k,\ell(k)}
                  +\rho_k\theta_{k+1}s^{k,\ell(k)}
                  -\hat r^{k}
          \right]\right)
      && [\text{by~\eqref{eq:galois}}]\\
    & = \hat{r}^k + \gamma \hat{b}^k +
      \rho_k \theta_{k+1} \big(r^{k,\ell(k)}-s^{k,\ell(k)}\big)
      \\            
    & = w^k + \rho_k \theta_{k+1} v^k 
       && [\text{by~\eqref{eq:cauchy}}]\\
                 & = w^{k}-\rho_k \dfrac{\inner{w^k-\tilde z^k}{v^k}}{\norm{v^k}^2}v^k,
      && [\text{by~\eqref{eq:cayley} and~\eqref{eq:cauchy}}]
\end{align*}
which establishes~\eqref{eq:err.hpe2} and thus completes the proof of the proposition.
\end{proof}

\noindent The following theorem states the asymptotic convergence properties
of Algorithm \ref{alg:ine_dr}, which are essentially direct consequences of
Proposition \ref{pr:2e1} and Theorem \ref{th:wc}.

\begin{theorem}[Convergence of Algorithm \ref{alg:ine_dr}]
\label{th:conv_dr}
Consider the sequences evolved by Algorithm~\ref{alg:ine_dr} with the
parameters $\alpha\in [0,1)$, $\overline{\rho}\in (0,2)$ and $\{\alpha_k\}$
satisfying the conditions~\eqref{eq:alpha} and~\eqref{eq:betatau} of
Theorem~\ref{th:wc}.  Then
\begin{itemize}
\item[\emph{(a)}] If the outer loop (over $k$) executes an infinite number of times, with each inner loop (over $\ell$) terminating in a finite number of iterations $\ell=\ell(k)$, then $\{s^k\}$ and $\{r^k\}$ both converge to some 
solution $x^*\in\R^n$ of \eqref{eq:mip_ab}, and $\{b^{k,\ell(k)}\}$ and
$\{b^k\}$ both converge to some $b^*\in B(x^*)$, with $\{a^{k,\ell(k)}\}$
converging to $-b^*\in A(x^*)$.
\item[\emph{(b)}]  If the outer loop executes only a finite number of times,
ending with $k=\bar k$, with the last invocation of the inner loop executing
an infinite number of times, then $\{s^{\bar k,\ell}\}_{\ell=1}^\infty$ and
$\{r^{\bar k,\ell}\}_{\ell=1}^\infty$ both converge to some solution
$x^*\in\R^n$ of \eqref{eq:mip_ab}, and $\{b^{\bar k,\ell}\}_{\ell=1}^\infty$ converges
to some $b^*\in B(x^*)$, with $\{a^{\bar k,\ell}\}_{\ell=1}^\infty$ converging to
$-b^*\in A(x^*)$.
\item[\emph{(c)}] If \emph{Algorithm \ref{alg:ine_dr}} stops with 
$s^{k,\ell}=r^{k,\ell}$, then $z^*:=s^{k,\ell}=r^{k,\ell}$ is a solution of \eqref{eq:mip_ab}.
\end{itemize}
\end{theorem}
\begin{proof}
\noindent
(a) For each $k\geq 0$, again let $\ell=\ell(k)$ be the index of inner
 iteration that first meets the inner-loop termination condition. Using
 Proposition \ref{pr:2e1}, \eqref{eq:abel}, the descriptions of
 algorithms~\ref{inertial.hpp} and~\ref{alg:ine_dr}, and Theorem \ref{th:wc},
 we conclude that there exists $z^*\in \R^n$ such that $0\in S_{\gamma, A,
 B}(z^*)$ and
\begin{align}
\label{eq:lanczos}
 z^k &= r^k+\gamma b^k\to z^* & 
 \tilde z^{k-1} &= r^{k}+\gamma b^{(k-1),\ell(k-1)}\to z^* &
 v^{k-1} &=s^{k}-r^{k}\to 0. 
\end{align}
From $0\in S_{\gamma, A, B}(z^*)$ and \eqref{eq:eck_bert} we obtain that
$x^*:=J_{\gamma B}(z^*)$ is a solution of \eqref{eq:mip_ab}. Moreover, it
follows from \eqref{eq:lanczos}, the inclusion in \eqref{eq:jordan},
\eqref{eq:abel}, and the continuity of $J_{\gamma B}$ that
\begin{align}
 \label{eq:hestenes02}
 s^{k}+\gamma b^{(k-1),\ell} &= v^{k-1} + \tilde z^{(k-1)} \to 0+z^*=z^* &
s^{k} &= J_{\gamma B}( s^{k}+\gamma b^{(k-1),\ell})\to 
 J_{\gamma B}(z^*) = x^*. 
\end{align}
We also have $r^k\to x^*$ since, from \eqref{eq:lanczos}, $s^k-r^k\to 0$.
Altogether, we have that $x^*$ is a solution of \eqref{eq:mip_ab} and
$\{s^k\}$ and $\{r^k\}$ both converge to $x^*$.
From \eqref{eq:hestenes02} we now have
\begin{align}
 \label{eq:hestenes}
 b^{k,\ell(k)}= \gamma^{-1}(s^{k+1}+\gamma b^{k,\ell(k)}-s^{k+1})
 \to \gamma^{-1}(z^*-x^*):=b^*.
\end{align}
From $x^*=J_{\gamma B}(z^*)$ we then obtain $b^*\in B(x^*)$. 
On the other hand, using the equation in \eqref{eq:tartaglia}, 
\eqref{eq:abel}, \eqref{eq:lanczos} and \eqref{eq:hestenes}
we find
\begin{align*}
 a^{k,\ell(k)}=\gamma^{-1}(s^{k+1}-r^{k+1}) - b^{k,\ell(k)}\to 0-b^*= -b^*.
\end{align*}
Using the above convergence result, that $r^{k,\ell(k)}=r^{k+1}\to x^*$,  the
inclusion in \eqref{eq:tartaglia}, and Lemma \ref{lm:simons}, we obtain that
$-b^*\in A(x^*)$. Finally, $b^k=\gamma^{-1}(z^k-r^k)\to
\gamma^{-1}(z^*-r^*)=b^*$.

\mgap

(b) First note that using \eqref{eq:ferreiro} we obtain $(s^{\bar
k,\ell},b^{\bar k,\ell}) = \mathcal{B}(\hat r^{\bar k}, \hat b^{\bar k},
\gamma, \hat s^{\bar k}, \hat b^{\bar k},\ell)$, which in view of Definition
\ref{def:bp} yields $(s^{\bar k,\ell}, b^{\bar k,\ell})\in B$, for all
$\ell\geq 1$, $s^{\bar k, \ell}+\gamma b^{\bar k, \ell}\to \hat r^{\bar
k}+\gamma \hat b^{\bar k}$, $s^{\bar k,\ell}\to x^*$, and $b^{\bar k, \ell}\to
b^*$, for some $x^*, b^*\in \R^n$.
Combining limits, we obtain that $\hat r^{\bar k}+\gamma \hat
 b^{\bar k}=x^*+\gamma b^*$. From Lemma \ref{lm:simons}, we also have $b^*\in
 B(x^*)$.
Now combining the limits with \eqref{eq:tartaglia} and the continuity of
$J_{\gamma A}$, we also find
\begin{align*}
 r^{\bar k,\ell} = J_{\gamma A} (s^{\bar k,\ell}-\gamma b^{\bar k, \ell})\to J_{\gamma A}(x^*-\gamma b^*)=:r^*
\end{align*}
and so
\begin{align}
  \label{eq:banach}
 a^{\bar k, \ell}=\gamma^{-1}\left(s^{\bar k,\ell}-r^{\bar k,\ell}\right) - b^{\bar k, \ell}\to
 \gamma^{-1}(x^*-r^*)-b^*=:a^*.
\end{align}
From the inclusion in \eqref{eq:tartaglia} and (again) Lemma \ref{lm:simons}
we obtain that $a^*\in A(r^*)$.  
On the other hand, using \eqref{eq:lagrange} and the hypothesis that the inner
loop executes an infinite number of times at iteration $k=\bar k$, we obtain,
for all $\ell \geq 1$, that 
\begin{align}
  \norm{s^{\bar k,\ell}+\gamma b^{\bar k,\ell}-(\hat r^{\bar k}+\gamma \hat b^{\bar k})}^2> \sigma^2\left( \norm{r^{\bar k,\ell}+\gamma b^{\bar k,\ell}-(\hat r^{\bar k}+\gamma \hat b^{\bar k})}^2 + \norm{s^{\bar k,\ell}-r^{\bar k,\ell}}^2\right).
\end{align}
Since the left-hand side of the above inequality converges to zero and the
right-hand side is nonnegative, the right-hand side also converges to zero and
in particular $s^{\bar k,\ell}-r^{\bar k,\ell}\to 0$. Since $s^{\bar
k,\ell}\to x^*$ and $r^{\bar k,\ell}\to r^*$, we conclude that $x^*=r^*$ and,
hence, from \eqref{eq:banach}, that $a^*=-b^*$.

\mgap

(c) If $s^{k,\ell}=r^{k,\ell}=:z^*$, then it follows from the inclusion in \eqref{eq:jordan} and \eqref{eq:tartaglia} that
$0=\gamma^{-1}(s^{k,\ell}-r^{k,\ell})= a^{k,\ell}+b^{k,\ell}\in A(r^{k,\ell})+B(s^{k,\ell})=(A+B)(z^*)$. 
%
%
\end{proof}

\section{A partially inexact relative-error inertial-relaxed \\ ADMM}
 \label{sec:admm_iner}

We now consider the convex optimization problem
\begin{align}
 \label{eq:opt}
  \min_{z\in \R^n}\left\{ f(x) + g(x)\right\}
\end{align}
where $f, g: \R^n\to (-\infty,\infty]$ are proper, convex and lower semicontinuous functions for which
$(\partial f+\partial g)^{-1}(0)\neq\emptyset$.

The alternating direction method of multipliers (ADMM)
\cite{gag.mer-dua.cma76, glo.mar-sur.fnm75} is a first-order algorithm for
solving \eqref{eq:opt} which has become popular over the last decade largely
due to its wide range of applications in data science (see, \emph{e.g.},
\cite{boy.par.chu-dis.ftml11}).
As applied to~\eqref{eq:opt}, one iteration of the ADMM may be described as:
\begin{align}
 \label{eq:spf}
& x^{k+1}\in \argmin_{x\in \R^n}\,\left\{f(x)+\inner{p^k}{x}+\dfrac{c}{2}\norm{x-z^k}^2\right\},\\
 \label{eq:spg}
& z^{k+1}\in \argmin_{z\in \R^n}\,\left\{g(z)-\inner{p^k}{z}+\dfrac{c}{2}\norm{x^{k+1}-z}^2\right\},\\
& p^{k+1} = p^k +c(x^{k+1}-z^{k+1}).
\end{align}

In many applications, the function $g$ is such that \eqref{eq:spg} has a
closed-form or otherwise straightforward solution (\emph{e.g.},
$g(\cdot)=\|\cdot\|_1$).  We consider situations in which this is the case,
but solving~\eqref{eq:spg} is more difficult and requires some form of
iterative process.
%
Eckstein and Yao~\cite[Section 6]{eck.yao-rel.mp17} proposed and studied the
asymptotic convergence of an inexact version of the ADMM tailored to such
situations:  at each iteration, \eqref{eq:spf} may be approximately solved
within a relative-error tolerance. This method is a special version of their
inexact relative-error Douglas-Rachford (DR) algorithm mentioned in Section
\ref{sec:dr}, as applied to the monotone inclusion problem
\begin{align}
 \label{eq:mip_fg}
 0 \in \partial f(x) + \partial g(x)
\end{align}
which is, in particular, a special case of \eqref{eq:mip_ab} with $A=\partial
f$ and $B=\partial g$. Problem \eqref{eq:mip_fg} is, under standard
qualification conditions, equivalent to \eqref{eq:opt}. Recall that we are
assuming $(\partial f+\partial g)^{-1}(0)\neq\emptyset$, \emph{i.e.}, that
\eqref{eq:mip_fg} admits at least one solution.

In this section, we propose and study the asymptotic behaviour of a
(partially) inexact relative-error \emph{inertial-relaxed} ADMM algorithm for
solving \eqref{eq:opt}. The proposed method, namely Algorithm
\ref{admm_inertial}, is a special version of Algorithm \ref{alg:ine_dr} when
applied to solving \eqref{eq:mip_fg} and may be viewed as an alternative to
the Eckstein-Yao approximate ADMM \cite{eck.yao-rel.mp17} that incorporates
inertial and relaxation effects to accelerate convergence.

To formalize the inexact solution process for the subproblems \eqref{eq:spf},
we introduce the notion of an $\mathcal{F}$-procedure~\cite{eck.yao-rel.mp17}.
First, we note that any instance of \eqref{eq:spf} can be posed slightly more
abstractly as
\begin{align}
   \label{eq:jordan03}
  \min_{x\in \R^n}\left\lbrace f(x)+\inner{p}{x}+\dfrac{c}{2}\norm{x - z}^2\right\rbrace
\end{align}
where $p,z\in \R^n$ and $c>0$.

\mgap

\begin{definition} [$\mathcal{F}$-procedure for solving \eqref{eq:jordan03}]
 \label{def:fp} 
 A $\mathcal{F}$--procedure
for (approximately) solving any instance of \eqref{eq:jordan03}  is a mapping 
$\mathcal{F} = (\mathcal{F}_1, \mathcal{F}_2):\R^n \times \R^n\times \R_{++}\times \R^n\times \N^*\to \R^n\times \R^n$ such that
if one lets $(x^\ell, y^\ell)=\mathcal{F}(p,z,c,\bar x,\ell)$ for all $\ell \in \N$ and any given $p,z, \bar x\in \R^n$ and $c>0$, then 
\begin{align}
  \label{eq:jordan04}
 \lim_{\ell\to \infty}\,y^\ell = 0\;\;\mbox{and}\;\; (\forall \ell \in \N)\;\; y^\ell\in \partial_x\left[f(x)+\inner{p}{x}+\dfrac{c}{2}\norm{x-z}^2\right]_{x=x^\ell}.
\end{align}
\end{definition}

\noindent Quoting \cite[Assumption 2]{eck.yao-rel.mp17}, ``the idea behind
this definition is that $\mathcal{F}(p,z,c,\bar x,\ell)$ is the
$\ell^{\text{th}}$ iterate produced by the $x$-subproblem solution procedure
with penalty parameter $c$, the Lagrange multiplier estimate $p^k$ equal to
$p$, and $z^k=z$, starting from the solution estimate $\bar x$''.
For the remainder of this section, we assume the following.

\begin{assumption} 
 \label{asp:fp}
There exists a $\mathcal{F}$--procedure (according to \emph{Definition
\ref{def:fp}}) for approximately solving any instance of \eqref{eq:jordan03}.
\end{assumption}

\noindent The next lemma shows that the $\mathcal{F}$-procedure is essentially
a form of $\mathcal{B}$--procedure (see Definition~\ref{def:bp}). Although the
proof essentially duplicates analysis
in~\cite{eck.yao-app.coap17,eck.yao-rel.mp17}, it is not presented as a
separate result there.  Therefore we include the proof in the interest of
rigor and completeness.

\begin{lemma}
 Let $\mathcal{F}(\cdot) = (\mathcal{F}_1 (\cdot), \mathcal{F}_2(\cdot))$ be a $\mathcal{F}$--procedure for solving \eqref{eq:jordan03}, where $\mathcal{F}_i: \R^n \times \R^n\times \R_{++}\times \R^n\times \N^*\to \R^n$, for $i=1,2$, and define
$\mathcal{B}:\R^n\times \R^n\times \R_{++}\times \R^n\times \R^n\times \N^*\to \R^n\times \R^n$ by
\begin{align}
 \label{eq:def.bf}
\mathcal{B}(r, b,\gamma,\bar s, \bar b, \ell) = \mathcal{F} (-b, r,\gamma^{-1}, \bar s, \ell) +
\left(0, \,b - \gamma^{-1}(\mathcal{F}_1(-b, r,\gamma^{-1}, \bar s, \ell)-r)\right).
\end{align}
Then, $\mathcal{B}$  is a $\mathcal{B}$--procedure (see \emph{Definition \ref{def:bp}}) for approximately solving \eqref{eq:jordan10} in which $s:=x$,  $B:=\partial f$, $\gamma=c^{-1}$,  $r:=z$ and $b:=-p$.
\end{lemma}
\begin{proof}
Assume that $(s^\ell, b^\ell)=\mathcal{B}(r, b,\gamma,\bar s, \bar b, \ell)$ for some
$r, b, \bar s, \bar b\in \R^n$, $\gamma>0$ and all $\ell \in \N^*$.
In view of \eqref{eq:def.bf} and the fact that $\mathcal{F}=(\mathcal{F}_1, \mathcal{F}_2)$ we have
\[
 (s^\ell, b^\ell) = \left(\mathcal{F}_1 (-b, r,\gamma^{-1}, \bar s, \ell), \mathcal{F}_2 (-b, r,\gamma^{-1}, \bar s, \ell)\right) +
\left(0, \,b - \gamma^{-1}(\mathcal{F}_1(-b, r,\gamma^{-1}, \bar s, \ell)-r)\right)
\]
and so, for all $\ell \in \N^*$,
\[
  (s^\ell, b^\ell - b + \gamma^{-1}(s^\ell-r)) = 
  \left(\mathcal{F}_1 (-b, r,\gamma^{-1}, \bar s, \ell), \mathcal{F}_2 (-b, r,\gamma^{-1}, \bar s, \ell) \right) =
 \mathcal{F} (-b, r,\gamma^{-1}, \bar s, \ell).
\]
Using the latter identity and the fact that $\mathcal{F}(\cdot)$ is a $\mathcal{F}$--procedure (see Definition \ref{def:fp}) we obtain 
%
\[
\lim_{\ell\to \infty}\, (\underbrace{b^\ell - b + \gamma^{-1}(s^\ell-r)}_{=:y^\ell})= 0\;\;\mbox{and}\;\; (\forall \ell \in \N)\;\; 
y^\ell \in 
\partial_x\left[f(x)-\inner{b}{x}+\dfrac{1}{2\gamma}\norm{x-r}^2\right]_{x=s^\ell}
\]
which, in particular, after some computations,  yields $(s^\ell, b^\ell)\in G(\partial f)$, i.e., $b^\ell \in \partial f(s^\ell)$ for all $\ell \in \N^*$. Using this fact and the definition of $y^\ell$ we find $s^\ell = (\gamma \partial f + I)^{-1}(r+\gamma(y^\ell+b))$, which in turn combined with the fact
that $\lim_{\ell \to \infty}\,y^\ell=0$ and the continuity of $J_{\gamma \partial f} := (\gamma \partial f + I)^{-1}$ implies that
$s^\ell \to J_{\gamma \partial f} (r+\gamma b)$. On the other hand, using the definition of $y^\ell$ (again) we 
also obtain $\gamma b^\ell + s^\ell = \gamma (y^\ell + b)+r$, which gives that $\{b^\ell\}$ is convergent and 
$\gamma b^\ell + s^\ell \to r+\gamma b$. Altogether, we proved that $(s^\ell, b^\ell)\in \partial f$, for all
$\ell\in \N^*$, that the sequence $\{(s^\ell, b^\ell)\}$ is convergent and  $s^\ell+\gamma b^\ell \to r+\gamma b$, which finishes the proof.
\end{proof}

Our inertial-relaxed inexact ADMM for solving~\eqref{eq:opt} is presented as
Algorithm~\ref{admm_inertial}.  
Before establishing its convergence, we make the following
remarks regarding this algorithm:

\begin{algfloat}[p!]
\noindent
\addtolength{\linewidth}{-2\fboxsep}%
\addtolength{\linewidth}{-2\fboxrule}%
\fbox{
\begin{minipage}{0.96\linewidth}
\begin{algorithm}
\label{admm_inertial}
{\bf Partially inexact relative-error inertial-relaxed ADMM for  \eqref{eq:opt}}
\end{algorithm}
\noindent Choose $c>0$, $0\leq \alpha,\sigma<1$ and $0<\underline{\rho}<\overline{\rho}<2$. \\
Initialize $(x^0,z^0,p^0)=(x^{-1},z^{-1},p^{-1})\in (\R^n)^3$. 

\mgap

\noindent {\bf for}\, $k=0,1,2,\dots$\,{\bf do}
\begin{enumerate}
\item[] Choose $\alpha_k\in [0,\alpha]$ and define 
  \begin{align}
    \label{eq:cartano02}
   (\hat x^k,\hat z^k,\hat p^k)=(x^k,z^k,p^k)+\alpha_k[(x^k,z^k,p^k)-(x^{k-1},z^{k-1},p^{k-1})]
   \end{align}
\item[] {\bf repeat}\,$\lbrace \mbox{{\bf for}}\,\ell=1,2,\dots\rbrace$
\begin{enumerate}
\item[] Improve the solution 
\begin{align}
\label{eq:jordan02}
x^{k+1}\approx \argmin_{x\in \R^n}\left\lbrace f(x)+\inner{\hat p^k}{x}+\dfrac{c}{2}\norm{x-\hat z^k}^2\right\rbrace
\end{align}
\item[] by setting 
\begin{align}
 \label{eq:ferreiro02}
(x^{k,\ell},y^{k,\ell})=\mathcal{F}(\hat p^k,\hat z^k,c,\hat x^k,\ell) 
\end{align}
(thus incrementally executing a step of the $\mathcal{F}$--procedure)
\item[] Define 
 \begin{align}
  \label{eq:paques}
   p^{k,\ell}=\hat p^k+c(x^{k,\ell}-\hat z^k)-y^{k,\ell}
 \end{align}
\item[] Exactly find
\begin{align}
 \label{eq:tartaglia02}
  z^{k,\ell} = \argmin_{z\in
  \R^n}\,\left\lbrace\,g(z)-\inner{p^{k,\ell}}{z}
     +\dfrac{c}{2}\norm{x^{k,\ell}-z}^2
  \right\rbrace
\end{align}
\end{enumerate}
\noindent {\bf until}  
\begin{align}
 \label{eq:lagrange02}
 \norm{y^{k,\ell}}^2\leq \sigma^2\left( \norm{p^{k,\ell}-\hat p^k-c(z^{k,\ell}-\hat z^k)}^2+ c^2\norm{x^{k,\ell}-z^{k,\ell}}^2\right)
\end{align}
\noindent {\bf if} 
$x^{k,\ell}=z^{k,\ell}$ then {\bf stop} \\
\noindent {\bf otherwise}, choose $\rho_k\in
[\,\underline{\rho},\overline{\rho}\,]$ and set
\begin{align}
  \label{eq:abel02}
 &x^{k+1}=x^{k,\ell},\quad z^{k+1}=z^{k,\ell}\\
 \label{eq:cayley02}
&\theta_{k+1}=\dfrac{\inner{c(\hat z^k-z^{k,\ell})-(\hat p^k-p^{k,\ell})}{x^{k,\ell}-z^{k,\ell}}}{c\norm{x^{k,\ell}-z^{k,\ell}}^2}\\
 \label{eq:galois02}
& p^{k+1}=\hat p^k+c\left[(1-\rho_k\,\theta_{k+1})z^{k+1}+\rho_k\,\theta_{k+1} x^{k+1}-\hat z^k\right]
\end{align}
\end{enumerate}
\vspace{-3ex}
\noindent
{\bf end for}
\mgap
\end{minipage}
} 
\end{algfloat}

\begin{itemize}
 \item[(i)] Similarly to Algorithm \ref{alg:ine_dr}, Algorithm \ref{admm_inertial} benefits from inertial and relaxation 
effects --- see \eqref{eq:cartano02} and \eqref{eq:galois02} --- as well as
from the relative error criterion \eqref{eq:lagrange02} allowing inexact
solution of the $f$-subproblem~\eqref{eq:jordan02}.
\item[(ii)] Algorithm \ref{admm_inertial} can be viewed as an inertial-relaxed
version of Algorithm 4 in \cite{eck.yao-rel.mp17}, but we emphasize that 
even without inertia or relaxation (that is, when $\alpha=0$ and
$\rho_k\equiv 1$) it differs from the latter algorithm since Algorithm 4 is
based on an approximate proximal point algorithm using an extragradient
``corrector'' step, while Algorithm \ref{admm_inertial} is instead based
indirectly on Algorithm~\ref{inertial.hpp}, an approximate proximal point
method using projective corrector steps.  In developing Algorithm
\ref{admm_inertial}, we also experimented with using extragradient correction,
but obtained better numerical performance from projective correction.

\item[(iii)]  The derivation of Algorithm~\ref{admm_inertial} mirrors that
in~\cite{eck.yao-rel.mp17}, except that the underlying convergence ``engine''
from~\cite{sol.sva-hyb.jca99} is replaced by Algorithm~\ref{inertial.hpp}.  It
should be noted that~\cite{eck.yao-app.coap17} provides a different way of
deriving approximate ADMM algorithms.  This approach results in different
approximate forms of the ADMM, allowing for both relative and absolute error
criteria, both of a practically verifiable form.  It is also possible that the
work in~\cite{sva-ful.preprint18} could lead to still more approximate forms
of the ADMM.
\end{itemize}



\begin{proposition}
 \label{pr:admm_e_dr}
For any given execution of Algorithm~\ref{admm_inertial}, define
%
\begin{align}
 \label{eq:broyden}
(s^k,  b^k , r^k)  &:= (x^k , - p^k , z^k) &  
(\hat s^k,  \hat b^k , \hat r^k)  &:= (\hat x^k , - \hat p^k , \hat z^k) \\
%
 \label{eq:broyden02}
(s^{k,\ell},  b^{k,\ell} , r^{k,\ell})  &:= (x^{k,\ell} , - p^{k,\ell} , z^{k,\ell}) &
 a^{k,\ell} &:= c(s^{k,\ell}-r^{k,\ell})-b^{k,\ell}
\end{align}
for all applicable $k$ and $l$.  Then these sequences conform to the recursions
\eqref{eq:cartano}-\eqref{eq:galois} in Algorithm \ref{alg:ine_dr} with
$\gamma = 1/c$, the $\mathcal{B}$-procedure \eqref{eq:def.bf}, and the maximal
monotone operators $A=\partial g$ and $B=\partial f$.
\end{proposition}
\begin{proof}
In view of \eqref{eq:broyden} and \eqref{eq:cartano02} we have
\begin{align*}
 (\hat s^k,\hat b^k,\hat r^k) = (\hat x^k , - \hat p^k , \hat z^k) 
           & = \left(x^{k}+\alpha_k (x^k-x^{k-1}), - p^{k}-\alpha_k (p^k-p^{k-1}), z^{k}+\alpha_k (z^k-z^{k-1})\right)\\
 &=(s^k,b^k,r^k)+\alpha_k \left[(s^k,b^k,r^k)-(s^{k-1},b^{k-1},r^{k-1})\right],
\end{align*}
which is identical to~\eqref{eq:cartano} in Algorithm \ref{alg:ine_dr}.
Fix $\gamma=1/c$.  Then \eqref{eq:ferreiro02}, Definition \ref{def:fp},
\eqref{eq:broyden} lead to
\begin{align}
  \label{eq:fletcher}
 x^{k,\ell} = \mathcal{F}_1(\hat p^k,\hat z^k,c,\hat x^k,\ell) = \mathcal{F}_1(-\hat b^k,\hat r^k,\gamma^{-1},\hat s^k,\ell)-\hat r^k).
 \end{align}
Combining \eqref{eq:broyden02}, \eqref{eq:paques}, \eqref{eq:ferreiro02},
\eqref{eq:fletcher}, \eqref{eq:broyden}, and \eqref{eq:def.bf}, we deduce that
\begin{align*}
  (s^{k,\ell}, b^{k,\ell}) & = (x^{k,\ell}, -p^{k,\ell}) \\
   & = (x^{k,\ell}, y^{k,\ell}) + (0, -\hat p^k - \gamma^{-1}(x^{k,\ell}-\hat z^k))\\
 & = \mathcal{F}(-\hat b^k,\hat r^k,\gamma^{-1},\hat s^k,\ell) 
                + (0, \hat b^k - \gamma^{-1}(\mathcal{F}_1(-\hat b^k,\hat r^k,\gamma^{-1},\hat s^k,\ell)-\hat r^k))\\
& = \mathcal{B}(\hat r^k, \hat b^k, \gamma, \hat s^k, \hat b^k,\ell),
\end{align*}
which yields \eqref{eq:jordan} and \eqref{eq:ferreiro}.
Note now that \eqref{eq:tartaglia02} is equivalent to the condition
$0\in \partial g (z^{k,\ell}) - p^{k,\ell} +c(z^{k,\ell}-x^{k,\ell})$, which, in view of \eqref{eq:broyden02}, is clearly equivalent to \eqref{eq:tartaglia} with $A=\partial g$. To prove \eqref{eq:lagrange}, note that from \eqref{eq:broyden}, \eqref{eq:broyden02}, \eqref{eq:paques} and \eqref{eq:lagrange02} we obtain
\begin{align*}
 \norm{s^{k,\ell}+\gamma b^{k,\ell}-(\hat r^k+\gamma \hat b^k)}^2 & =  \norm{\gamma\, y^{k,\ell}}^2\\
 & \leq \gamma^2\sigma^2\left( \norm{p^{k,\ell}-\hat p^k-c(z^{k,\ell}-\hat z^k)}^2+ c^2\norm{x^{k,\ell}-z^{k,\ell}}^2\right)
\end{align*}
which in view of \eqref{eq:broyden} and \eqref{eq:broyden02} is equivalent to
\eqref{eq:lagrange}. Finally, similar reasoning establishes that
\eqref{eq:abel}-\eqref{eq:galois} are equivalent to
\eqref{eq:abel02}-\eqref{eq:galois02}.
\end{proof}


\begin{theorem}[Convergence of Algorithm \ref{admm_inertial}]
\label{th:conv_admm}
%
Consider any execution of Algorithm~\ref{admm_inertial} for which $\alpha\in
[0,1)$, $\overline{\rho}\in (0,2)$, and $\{\alpha_k\}$ satisfy conditions
\eqref{eq:alpha} and \eqref{eq:betatau} of Theorem \ref{th:wc}.
Then:
\begin{itemize}
\item[\emph{(a)}] If for each $k\geq 0$ the outer loop (over $k$) executes an
infinite number of times, with each inner loop (over $\ell$) terminating in a
finite number of iterations $\ell=\ell(k)$, then $\{x^k\}$ and $\{z^k\}$ both
converge to some $x^*\in\R^n$ solution of
\eqref{eq:mip_fg}, and $\{p^k\}$ converges to some $p^*\in \partial g(x^*)$
such that $-p^*\in \partial f(x^*)$.
\item[\emph{(b)}]  If the outer loop executes only a finite number of times, ending with $k=\bar k$, with the last invocation of the inner loop executing an infinite number of times, then $\{x^{\bar k,\ell}\}_{\ell}$ and $\{z^{\bar k,\ell}\}_{\ell}$ both converge to some $x^*\in\R^n$ solution of \eqref{eq:mip_fg}, and $\{p^{\bar k,\ell}\}_{\ell}$ converges to some $p^*\in \partial g(x^*)$ 
such that $-p^*\in \partial f(x^*)$.
\item[\emph{(c)}] If \emph{Algorithm \ref{admm_inertial}} stops with either $p^{k,\ell}-\hat p^k=c(z^{k,\ell}-\hat z^k)$ or $x^{k,\ell}=z^{k,\ell}$ then $x^*:=x^{k,\ell}=z^{k,\ell}$ is a solution of \eqref{eq:mip_fg}.
\end{itemize}
\end{theorem}
\begin{proof}
The result follows from immediately by combining Proposition \ref{pr:admm_e_dr},
Theorem \ref{th:conv_dr}, and the definitions of Algorithms \ref{alg:ine_dr}
and \ref{admm_inertial}.
\end{proof}

\section{Numerical experiments}
 \label{sec:exp}

This section describes numerical experiments on the LASSO and logistic regression problems, which are both instances of the minimization problem \eqref{eq:opt}. We tested the following algorithms:  the inexact relative-error ADMM \emph{admm\_primDR} from
 \cite{eck.yao-rel.mp17}; the relative-error method \emph{relerr} from \cite{eck.yao-app.coap17};  Algorithm \ref{admm_inertial} from this paper, which we denote as
\emph{admm\_primDR\_relx\_in}; the absolute-error aproximate ADMM \emph{absgeom} discussed in \cite{eck.yao-rel.mp17} and a backtraking variant of \emph{FISTA} \cite{MR2486527} (also discussed in \cite{eck.yao-rel.mp17}). We implemented all algorithms in MATLAB, and, analogously
to \cite{eck.yao-rel.mp17}, we used the following condition to terminate the outer loop:
\begin{align}
 \mbox{dist}_{\infty}\left(0,\partial_x[f(x)+g(x)]_{x=x^k}\right)\leq \epsilon,
\end{align}
where 
$\mbox{dist}_{\infty}(t,S):=\inf\{\|t-s\|_\infty\,|\,s\in S\}$, and $\epsilon>0$ is a tolerance parameter set to $10^{-6}$.

Moreover, in our implementation of Algorithm \ref{admm_inertial} from this paper, we replaced the error condition \eqref{eq:lagrange02} with the stronger condition
\begin{align}
 \label{eq:lagrange03}
 \norm{y^{k,\ell}} &\leq 
 \sigma \max\big\{\norm{p^{k,\ell}-\hat p^k-c(z^{k,\ell}-\hat z^k)}, 
                  c\norm{x^{k,\ell}-z^{k,\ell}}\big\},
\end{align}
which we empirically found to yield better numerical performance.

\subsection{Numerical experiments on the LASSO problem}
\label{subsec:ne}

In this subsection, we report numerical experiments on the LASSO
problem~\cite{tib-las.jrss96}
\begin{align}
 \label{eq:lasso}
 \min_{x\in \R^n}\,\dfrac{1}{2}\norm{Ax-b}^2 + \nu\|x\|_1,
\end{align}
where $A\in \R^{m\times n}$, $b\in \R^m$ and $\nu>0$, which is an instance of \eqref{eq:opt} with
$f(x):=(1/2)\norm{Ax-b}^2$ and $g(x):=\nu\norm{x}_1$.  For the data $A$ and
$b$, we used the same (non-artificial) datasets as in~\cite{eck.yao-rel.mp17}.
%
%
%
%
%

We tested three algorithms for solving \eqref{eq:lasso}:
\begin{itemize}
 \item The inexact relative-error ADMM \emph{admm\_primDR} from
 \cite{eck.yao-rel.mp17}. For this algorithm, we used the same parameter
  values as in~\cite{eck.yao-rel.mp17}, namely $\sigma=0.99$ and $c=1$ (except
  for the PEMS problem instance, for which $c=20$).  
%
%
\item The relative-error algorithm \emph{relerr} from \cite{eck.yao-app.coap17}. 
We also used $\sigma=0.99$, $c=1$ (for all problem instances except PEMS, which we used $c=20$). 
For this set of LASSO
  problems, the experiments in~\cite{eck.yao-app.coap17,eck.yao-rel.mp17} already
  show \emph{admm\_primDR} to outperform the algorithms
  of~\cite{eck.yao-app.coap17}, as well as FISTA~\cite{MR2486527}.
\item Algorithm \ref{admm_inertial} from this paper which we denote as
\emph{admm\_primDR\_relx\_in}. 
 We used the parameter settings $\alpha_k\equiv
\alpha = 0.18966$, $\beta=0.18976$ and 
$\rho_k\equiv \underline{\rho}=\overline{\rho}=1.4882$
--- see conditions \eqref{eq:alpha} and \eqref{eq:betatau} and Figure
    \ref{fig02}. We also set $\sigma=0.99$ and $c=1$ (except for the PEMS
    problem instance, for which $c=20$).
\end{itemize}

We implemented all of the 
algorithms in MATLAB, using a conjugate gradient procedure to
approximately solve the subproblems corresponding to $f(x) =
(1/2)\norm{Ax-b}^2$, exactly as in~\cite{eck.yao-rel.mp17}.
Table~\ref{tab:outer} shows number of outer iterations, Table~\ref{tab:inner}
shows the total number of inner (conjugate gradient) iterations, and
Table~\ref{tab:runtime02} shows runtimes in seconds. 
Figure~\ref{fig03} shows the same results graphically.
In each table, the smallest value in each row
appears in bold.  In terms of runtime, the new algorithm
outperforms that of~\cite{eck.yao-rel.mp17} for all problem except the
\emph{finance1000} instance.

\addtolength{\tabcolsep}{-2pt} 

\begin{table}
\begin{center}
\caption{LASSO outer iterations; $\alpha=0.18966$, $\beta=0.18976$ and $\bar\rho=1.4882$} 
\label{tab:outer}
\begin{tabular}{lrcrcrcrcrr} \hline \\
   Problem & & relerr &&  primDR & &
    primDR\_relx\_in & &  $\frac{iteration 3}{iteration 1}$ & &
   $\frac{iteration 3}{iteration 2}$ \\
   & & {\tiny $(iteration 1)$}  & & {\tiny $(iteration 2)$} & &  {\tiny $(iteration 3)$}  & & & & \\ \hline \\
   Ball64\_singlepixcam & & 280 & & 278 & & \textbf{123} & & 0.439 & & 0.442 \\ 
    Logo64\_singlepixcam & & 283 & & 282 & & \textbf{139} & & 0.491 & & 0.493 \\
    Mug32\_singlepixcam & & 153 & & 153 & &  \textbf{136} & & 0.888 & & 0.888 \\
    Mug128\_singlepixcam & & 920 & & 914 & & \textbf{435} & & 0.473 & & 0.476\\
    finance1000 & & \textbf{974} & & 1709 & & 1079 & & 1.107 & & 0.631 \\
    PEMS & & 3354 & & 3648 & & \textbf{1088} & & 0.324 & & 0.298 \\
    Brain & & 1855 & & 2295  & & \textbf{1219} & & 0.657 & & 0.531\\
    Colon & & 450 & & 482 & & \textbf{256} & & 0.568 & & 0.531\\
    Leukemia & & 675 & & 774 & & \textbf{424} & & 0.628 & & 0.547\\
    Lymphoma & & 908 & & 925 & & \textbf{482} & & 0.531 & & 0.521\\
    Prostate & & 1520 & & 1739 & & \textbf{998} & & 0.656 & & 0.574 \\
    srbct & & 426 & & 401 & & \textbf{221} & & 0.519 & & 0.551 \\ \hline
    Geometric mean & & 692.06 & &  761.02 & & \textbf{399.85} & & 0.577 & & 0.525 \\ \hline \\
\end{tabular}
\end{center}
\end{table}
\FloatBarrier

\begin{table}
\begin{center}
\caption{LASSO total inner iterations; $\alpha=0.18966$, $\beta=0.18976$ and $\bar\rho=1.4882$} 
\label{tab:inner}
\begin{tabular}{lrcrcrcrcrr}\hline \\
Problem & & relerr & & primDR & & primDR\_relx\_in & & {$\frac{iteration 3}{iteration 1}$} & &  {$\frac{iteration 3}{iteration 2}$} \\
& & {\tiny $(iteration 1)$} & & {\tiny $(iteration 2)$} & & {\tiny $(iteration 3)$}  & & & &  \\  \hline \\
  Ball64\_singlepixcam & & 603 & & 382 & & \textbf{191} & & 0.316 & & 0.500 \\
   Logo64\_singlepixcam & & 621 & & 369 & & \textbf{212} & & 0.341 & & 0.574 \\
    Mug32\_singlepixcam & & 998 & & 307 & & \textbf{302} & & 0.303 & & 0.984\\
    Mug128\_singlepixcam & & 1214 & & 1046 & & \textbf{488} & & 0.402 & & 0.466   \\
    finance1000 & & 18944 & & \textbf{7852} & & 9737 & & 0.514 & & 1.240\\
    PEMS & & 85858 & & 9318 & &  \textbf{9235} & & 0.107 & & 0.991 \\
    Brain & & 24612 & & \textbf{7116} & & 7655 & & 0.311 & & 1.075\\
    Colon & & 5847 & & \textbf{1401} & &  1461 & & 0.249 & & 1.042\\
    Leukemia & & 7888 & & \textbf{2321} & & 2543 & & 0.322 & & 1.095\\
    Lymphoma & & 15266 & & 3179 & & \textbf{3083} & & 0.202 & & 0.969\\
    Prostate & & 20615 & & \textbf{5193} & & 6629 & & 0.321 & & 1.276\\
    srbct & & 6213 & & 1505 & & \textbf{1334} & & 0.215 & & 0.886 \\ \hline
    Geometric mean & & 5859.43 & & 1876.32 & & \textbf{1652.97} & & 0.282 & & 0.880\\ \hline \\
\end{tabular}
\end{center}
\end{table}
\FloatBarrier

\begin{table}
\begin{center}
\caption{LASSO runtimes in seconds; $\alpha=0.18966$, $\beta=0.18976$ and $\bar\rho=1.4882$} 
\label{tab:runtime02}
\begin{tabular}{lrcrcrcrcrr}\hline \\
Problem & & relerr & & primDR & & primDR\_relx\_in & & {$\frac{time 3}{time 1}$} & & {$\frac{time 3}{time 2}$} \\
& & {\tiny $(time 1)$} & & {\tiny $(time 2)$} & & {\tiny $(time 3)$} & & & & \\ \hline \\
Ball64\_singlepixcam & & 11.02 & & 7.86 & & \textbf{3.75} & & 0.341 & & 0.477 \\
    Logo64\_singlepixcam & & 11.37 & & 7.62 & & \textbf{4.04} & & 0.355 & & 0.531  \\
    Mug32\_singlepixcam & & 1.07 & & 0.51 & & \textbf{0.43} & & 0.374 & & 0.862 \\
    Mug128\_singlepixcam & & 248.38 & & 218.08 & & \textbf{101.17} & & 0.407 & & 0.464 \\
    finance1000 & & 805.17 & & \textbf{327.56} & & 347.97 & & 0.432 & & 1.062 \\
    PEMS & & 7546.11 & & 1092.16 & & \textbf{988.12} & & 0.131& & 0.905 \\
    Brain & & 13.59 & & 5.94 & & \textbf{5.53} & & 0.407 & & 0.929 \\
    Colon & & 1.56 & & 0.45 & & \textbf{0.28} & & 0.179 & & 0.620 \\
    Leukemia & & 4.24 & & 2.23 & & \textbf{1.59} & & 0.375 & & 0.717  \\
    Lymphoma & & 7.18 & & 2.63 & & \textbf{2.03} & & 0.283 & & 0.773 \\
    Prostate & & 33.21 & & 13.15 & & \textbf{11.88} & & 0.357 & & 0.904\\
    srbct & & 1.83 & & 0.42 & & \textbf{0.35} & & 0.192 & & 0.847  \\ \hline
    Geometric mean & & 21.13 & & 8.75 & & \textbf{6.41} & & 0.303 & & 0.733 \\ \hline 
\end{tabular}
\end{center}
\end{table}
\FloatBarrier

\vspace{0.6cm}
\begin{figure}
    \centering
    \caption{Comparison of performance in LASSO problems}
    \label{fig03}
    \vspace{0.2cm}
    \begin{minipage}{\linewidth}
    \centering
    \subfloat[][LASSO outer iterations]{
        \begin{tikzpicture}[scale=.55]
    \begin{axis}[
    width  = 13cm,
    height = 8cm,
    major x tick style = transparent,
    ybar=0.2mm,
    bar width=1.5mm,
    ymajorgrids = true,
    ylabel={Outer iterations},
    symbolic x coords={ Ball64, Logo64, Mug32,  Mug128, Finance, PEMS, Brain,  Colon, Leukemia,  Lymphoma, Prostate, Srbct},
    xtick = data,
    x tick label style = {rotate=90},
    scaled y ticks= base 10:-3,
    ymin=0,ymax=3.8e3,
    legend cell align= left,
    legend style={
        at={(0.02,0.72)},
        anchor= south west,
        column sep=1ex},]
    \addplot [yellow, fill=yellow!90] table[x=simname,y=relerr] {\dispdata};
    \addplot [blue, fill=blue!90] table[x=simname,y=primDR]    {\dispdata};
    \addplot [red, fill=red!90] table[x=simname,y=PrimDR]  {\dispdata};
    \legend{\small relerr, \small primDR, \small primDR\_relx\_in};
  \end{axis}
\end{tikzpicture} } \hspace{0.5cm}
    \subfloat[][LASSO total inner iterations]{
         \begin{tikzpicture}[scale=.55]
    \begin{axis}[
    width  = 13cm,
    height = 8cm,
    major x tick style = transparent,
    ybar=0.2mm,
    bar width=2mm,
    ymajorgrids = true,
    ylabel={Cumulative inner iterations},
    symbolic x coords={Ball64, Logo64, Mug32, Mug128, Colon, Leukemia, Lymphoma, Prostate, Srbct},
    xtick = data,
    x tick label style = {rotate=90},
    scaled y ticks= base 10:-3,
    ymin=0,ymax=21500,
    ytick style={draw=none},
    legend cell align=left,
    legend style={
        at={(0.02,0.72)},
        anchor= south west,
        column sep=1ex
    },
    ]

    \addplot [yellow, fill=yellow!90] table[x=name,y=rel] {\dispdata};
    \addplot [blue, fill=blue!90] table[x=name,y=prim]    {\dispdata};
    \addplot [red, fill=red!90] table[x=name,y=Prim]  {\dispdata};
    \legend{\small relerr, \small primDR, \small primDR\_relx\_in};
    \end{axis}
\end{tikzpicture}}
    \end{minipage}\par\medskip
    
    \begin{minipage}{\linewidth}
    \centering
    \subfloat[][LASSO total inner iterations]{
         \begin{tikzpicture}[scale=.55]
     \begin{axis}[
    width  = 13cm,
    height = 8cm,
    major x tick style = transparent,
    ybar=0.2mm,
    bar width=2mm,
    ymajorgrids = true,
    ylabel={Cumulative inner iterations},
    symbolic x coords={Finance, PEMS, Brain},
    xtick = data,
    x tick label style = {rotate=90},
    scaled y ticks= base 10:-3,
    ymin=0,ymax=88500,
    ytick style={draw=none},
    legend cell align=left,
    legend style={
        at={(0.02,0.72)},
        anchor= south west,
        column sep=1ex
    },
    ]

    \addplot [yellow, fill=yellow!90] table[x=A,y=B] {\dispdata};
    \addplot [blue, fill=blue!90] table[x=A,y=C]    {\dispdata};
    \addplot [red, fill=red!90] table[x=A,y=D]  {\dispdata};
    \legend{\small relerr, \small primDR, \small primDR\_relx\_in};
\end{axis}
\end{tikzpicture}} \hspace{0.5cm}
    \subfloat[][LASSO runtimes in seconds]{
         \begin{tikzpicture}[scale=.55]
    \begin{axis}[
    width  = 13cm,
    height = 8cm,
    major x tick style = transparent,
    ybar=0.2mm,
    bar width=2mm,
    ymajorgrids = true,
    ylabel={Runtimes in seconds},
    symbolic x coords={Mug128, Finance, PEMS},
    xtick = data,
    x tick label style = {rotate=90},
    scaled y ticks= base 10:-3,
    ymin=0,ymax=8000,
    legend cell align= left,
    legend style={
        at={(0.02,0.72)},
        anchor= south west,
        column sep=1ex
    },
    ]
    \addplot [yellow, fill=yellow!90] table[x=E,y=F] {\dispdata};
    \addplot [blue, fill=blue!90] table[x=E,y=G]    {\dispdata};
    \addplot [red, fill=red!90] table[x=E,y=G]  {\dispdata};
    \legend{\small relerr, \small primDR, \small primDR\_relx\_in};
  \end{axis}
\end{tikzpicture}}
    \end{minipage}\par\medskip
    
    \begin{minipage}{0.98 \linewidth}
    \centering
   \subfloat[][LASSO runtimes in seconds]{
         \begin{tikzpicture}[scale=0.55]
    \begin{axis}[
    width  = 13cm,
    height = 8cm,
    major x tick style = transparent,
    ybar=0.2mm,
    bar width=2mm,
    ymajorgrids = true,
    ylabel={Runtimes in seconds},
    symbolic x coords={Ball64, Logo64, Mug32, Brain, Colon, Leukemia, Lymphoma, Prostate, Srbct},
    xtick = data,
    x tick label style = {rotate=90},
    scaled y ticks = false,
    ymin=0,ymax=35,
    ytick style={draw=none},
    legend cell align=left,
    legend style={
        at={(0.02,0.72)},
        anchor= south west,
        column sep=1ex
    },
    ]

    \addplot [yellow, fill=yellow!90] table[x=Name,y=err] {\dispdata};
    \addplot [blue, fill=blue!90] table[x=Name,y=DR]    {\dispdata};
    \addplot [red, fill=red!90] table[x=Name,y=Relx]  {\dispdata};
    \legend{\small relerr, \small primDR, \small primDR\_relx\_in};
    \end{axis}
   \end{tikzpicture}}
    \end{minipage}
    \end{figure}
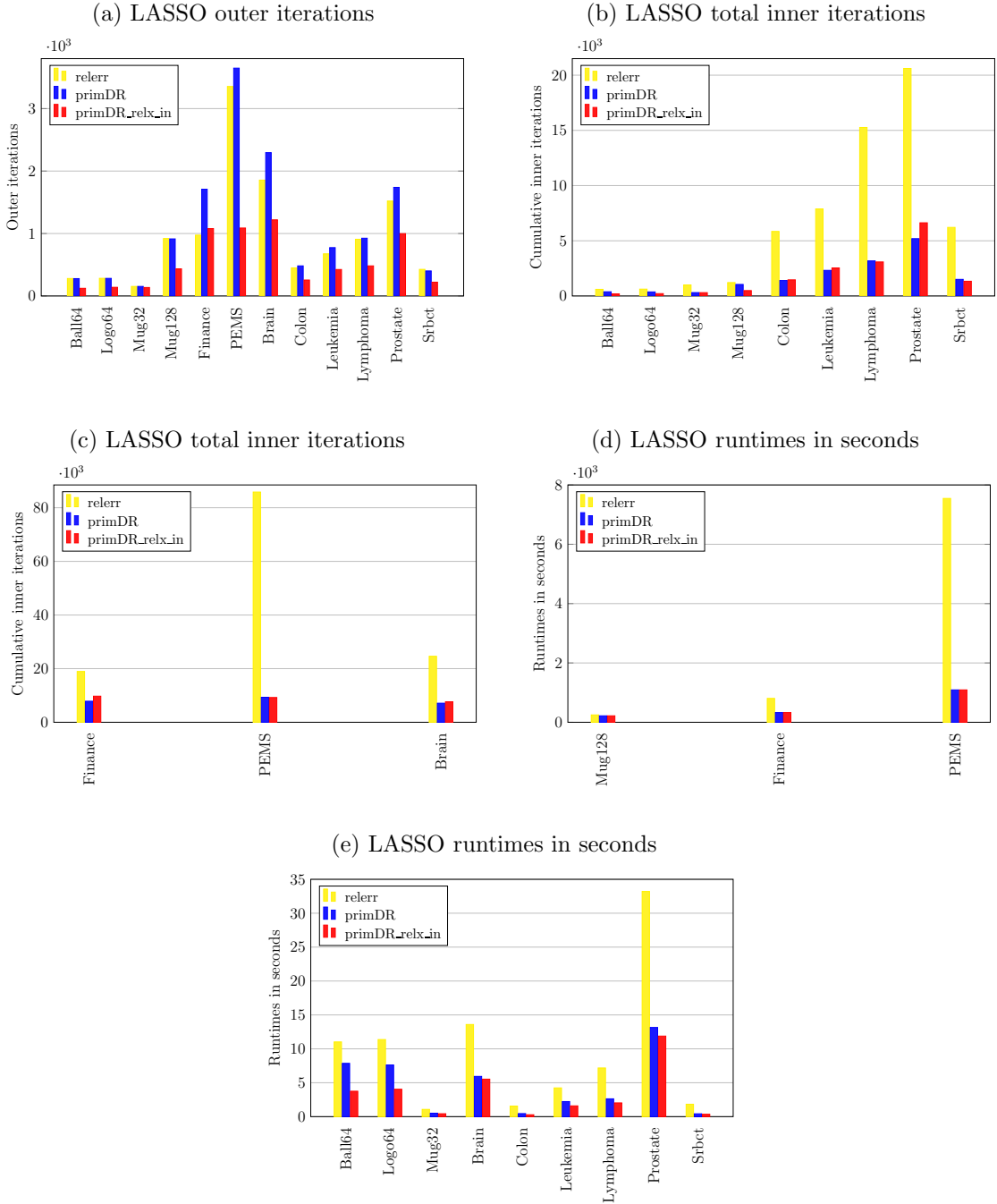
    \FloatBarrier

\subsection{Numerical experiments on logistic regression problems}
\label{subsec:log}

This section describes numerical experiments on the $\ell_1$--regularized logistic regression problem \cite{fra-ele.mi05, ng-fea.icml04}
\begin{align}
 \label{pr:lr}
  \min_{(w,v)\in \R^{n-1}\times \R}\,\sum_{i=1}^q\,\log\left(1+\mbox{exp}(-b_i(a_i^Tw+v))\right)+\nu\norm{w}_1,
\end{align}
using a training dataset consisting of $q$ pairs $(a_i,b_i)$,
where $a_i\in \R^{n-1}$ is a feature vector, $b_i\in \{-1,+1\}$ is the
corresponding label, $w\in \R^{n-1}$ represents a weighting of the feature and
$v\in \R$ reresents a kind of bias. Problem \eqref{pr:lr} is clearly a special
instance of \eqref{eq:opt} with $x=(v,w)$ and
\begin{align}
 \label{eq:fg.lr}
 f(v,w):=\sum_{i=1}^q\,\log\left(1+\mbox{exp}(-b_i(a_i^Tw+v))\right)\;\;\;\mbox{and}\;\;\;g(v,w):=\nu\norm{w}_1.
\end{align}
We considered four standard cancer DNA microarray
non-artificial datasets from \cite{det.buh-fin.jma04} (also
used in \cite[Subsection 7.2]{eck.yao-rel.mp17}) and tested five algorithms:
\emph{absgeom}, \emph{relerr}, \emph{admm\_primDR}, \emph{FISTA} and
\emph{admm\_primDR\_relx\_in}. For \emph{relerr} and \emph{admm\_primDR}
algorithms we used the same parameter values as in Subsection \ref{subsec:ne};
for \emph{admm\_primDR\_relx\_in} we used the parameter settings
$\alpha_k\equiv
\alpha = 0.1$, $\beta=0.1001$ 
and $\rho_k\equiv \underline{\rho}=\overline{\rho}=1.7606$ --- see conditions
\eqref{eq:alpha} and \eqref{eq:betatau} and Figure \ref{fig02}. We also
set $\sigma=0.99$ and $c=1$.

Analogously to \cite{eck.yao-rel.mp17}, we used an L-BFGS
procedure to approximately solve the subproblems
corresponding to $f(\cdot)$ from \eqref{eq:fg.lr}. Tables
\ref{tab:outer04}, \ref{tab:inner05} and \ref{tab:time06} show outer
iterations, total inner iterations and runtimes, respectively. These results
are also graphically summarized in Figure \ref{fig04}.
The new algorithm has the best aggregate performance by all
measures, and the best run time for all the datesets.

\begin{table}
\centering
\caption{Outer iterations for logistic regression problems.} 
\label{tab:outer04}
\begin{tabular}{lrcrcrcc}\hline  \\
Problem & & absgeom & & relerr & & primDR &  primDR\_relx\_in  \\
& & {\tiny $(iteration 1)$} & & {\tiny $(iteration 2)$} & & {\tiny $(iteration 3)$} & {\tiny $(iteration 4)$}  \\ \hline \\
    Colon & & 2666 & & 2145 & & 1979 &  \textbf{1578}   \\
    Leukemia & & 1662 & & 1116  & & 922 &  \textbf{788}  \\
    Prostate & & 1936 & & 1583 & & 1677 & \textbf{1198} \\
    Arcene & & 419 & & \textbf{276} & & 359 &  290 \\  \hline
    Geometric mean & & 1376.91  & & 1011.28 & & 1023.76 &  \textbf{810.72} \\ \hline \hline \\
    Problem & & $\frac{iteration 4}{iteration 1}$ & & $\frac{ iteration 4}{iteration 2}$ & & $\frac{iteration 4}{iteration 3}$ &  $\frac{iteration 2}{iteration 3}$  \\ 
    & &  & &  & &  & \\ \hline \\
    Colon & & 0.5919 & & 0.7356 & & 0.7974 &  1.0839  \\
    Leukemia & & 0.4741 & & 0.7061 & & 0.8546 & 1.2104  \\
    Prostate & & 0.6188 & & 0.7568 & & 0.7144 & 0.9439 \\
    Arcene & & 0.6921 & & 1.0507 & & 0.8078 & 0.7688 \\  \hline
    Geometric mean & & 0.5887 & & 0.8016 & & 0.7924 & 0.9849 \\ \hline 
\end{tabular}
\end{table}

\vspace{0.3cm}
\begin{table}
\begin{center}
\centering
\caption{Total inner iterations for logistic regression problems.} 
\label{tab:inner05}
\begin{tabular}{lrcrcrcrcc}\hline \\
Problem & & absgeom & & relerr & & primDR & & FISTA & primDR\_relx\_in  \\
& & {\tiny $(iteration 1)$} & & {\tiny $(iteration 2)$} & & {\tiny $(iteration 3)$} & & {\tiny $(iteration 4)$} & {\tiny $(iteration 5)$} \\ \hline \\
    Colon & & 20612 & & 23919 & & 21697 & & 26247  & \textbf{8283}  \\
    Leukemia & & 7715 & & 12086 & & 11625 & & 6536 &  \textbf{4448}  \\
    Prostate & & 18901 & & 27505 & & 24548 & & 13730 & \textbf{10997} \\
    Arcene & & \textbf{780} & & 3236 & & 3589 & & 4648 & 1450\\  \hline
    Geometric mean & & 6958.73  & & 12665.18 & & 12209.43 & & 10228.97 & \textbf{4923.21} \\ \hline \hline \\
    Problem & & $\frac{iteration 5}{iteration 1}$ & & $\frac{iteration 5}{iteration 2}$ & & $\frac{iteration 5}{iteration 3}$ & & $\frac{iteration 5}{iteration 4}$ & $\frac{iteration 4}{iteration 1}$ \\ 
    & &  & &  & &  & &  & \\ \hline \\
    Colon & & 0.4018 & & 0.3463 & & 0.3817 & & 0.3156 & 0.9499  \\
    Leukemia & & 0.5765 & & 0.3681 & & 0.3826  & & 0.6805 & 0.6636 \\
    Prostate & & 0.5818 & & 0.3998 & & 0.4479 & & 0.8009 & 0.7699 \\
    Arcene & & 1.8589 & & 0.4481 & & 0.4041 & & 0.3119 & 0.2173 \\  \hline
    Geometric mean & & 0.7074 & & 0.4032 & & 0.4032 & & 0.4813 & 0.5699 \\ \hline 
\end{tabular}
\end{center}
\end{table}

\begin{table}
\centering
\caption{Logistic regression runtimes in seconds.}
\label{tab:time06}
\begin{tabular}{lrcrcrcrcc}\hline \\
 Problem & & absgeom & & relerr & & primDR & & FISTA & primDR\_relx\_in  \\
& & {\tiny $(time 1)$} & & {\tiny $(time 2)$} & & {\tiny $(time 3)$} & & {\tiny $(time 4)$} &  {\tiny $(time 5)$} \\ \hline \\
    Colon & & 182.3601 & & 36.5207 & & 91.5726 &  &73.2987  & \textbf{12.8243}  \\
    Leukemia & & 112.7412 & & 105.4221 & & 241.1378 & & 60.9476 & \textbf{23.0547}   \\
    Prostate & & 342.1609 & & 719.6731 & & 850.8159 & & 206.3883 & \textbf{128.6972} \\
    Arcene & & 122.7208 & & 312.1101 & & 370.9415 & & 184.3489 & \textbf{46.1276}\\  \hline
    Geometric mean & & 171.41 & & 224.11  & & 288.93 & & 114.18 & \textbf{36.39}  \\\hline \hline \\
    Problem & & $\frac{time 5}{time 1}$ & & $\frac{time 5}{time 2}$ & & $\frac{time 5}{time 3}$ & & $\frac{time 5}{time 4}$ & $\frac{time 4}{time 3}$ \\ 
    & &  & &  & &  & & & \\\hline \\
    Colon & & 0.0703 & & 0.1203 & & 0.1401 & & 0.1749 & 0.8003    \\
    Leukemia & & 0.2045 & & 0.2186 & & 0.0956 & & 1.0215 & 0.2527   \\
    Prostate & & 0.3761 & & 0.1788 & & 0.1513 & & 0.6236 & 0.2426  \\
    Arcene & & 0.3759 & & 0.1478 & & 0.1244 & & 0.2502 & 0.4969 \\  \hline
    Geometric mean & & 0.2123 & & 0.1623 & & 0.1259 & & 0.3187 & 0.3951 \\ \hline 
\end{tabular}
\end{table}

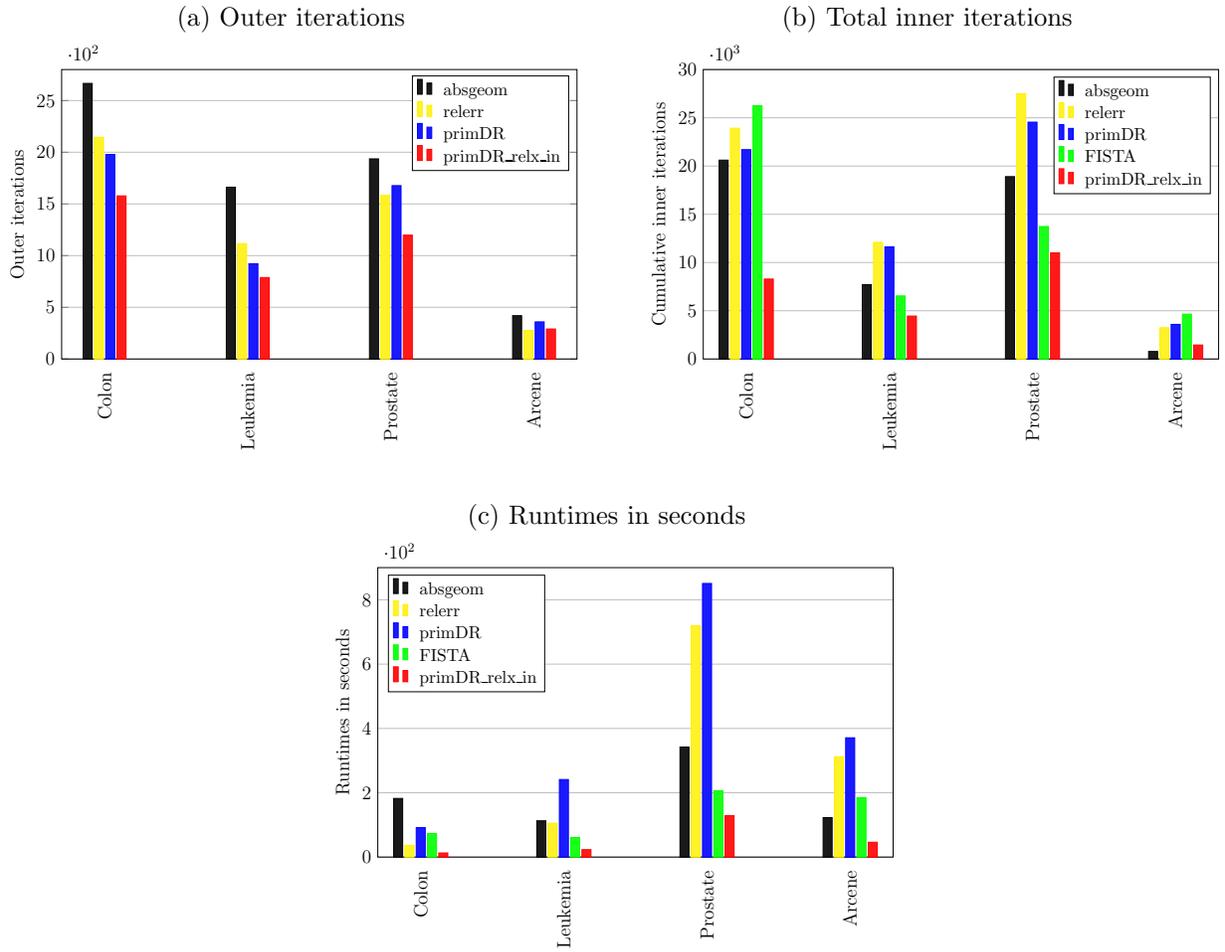
\begin{figure}
    \centering
    \caption{Comparison of performance in logistic regression problems}
    \label{fig04}
    \vspace{0.2cm}
    \begin{minipage}{\linewidth}
    \centering
    \subfloat[][Outer iterations]{
        \begin{tikzpicture}[scale=.6]
    \begin{axis}[
    width  = 13cm,
    height = 8cm,
    major x tick style = transparent,
    ybar=0.5mm,
    bar width=2mm,
    ymajorgrids = true,
    ylabel={Outer iterations},
    symbolic x coords={ Colon, Leukemia, Prostate, Arcene},
    xtick = data,
    x tick label style = {rotate=90},
    scaled y ticks= base 10:-2,
    ymin=0,ymax=2.8e3,
    legend cell align= left,
    legend style={
        at={(0.68,0.65)},
        anchor= south west,
        column sep=1ex},]
    \addplot [black, fill=black!90] table[x=I,y=J] {\dispdata};
    \addplot [yellow, fill=yellow!90] table[x=I,y=K]    {\dispdata};
    \addplot [blue, fill=blue!90] table[x=I,y=L]  {\dispdata};
    \addplot [red, fill=red!90] table[x=I,y=M]  {\dispdata};
    \legend{\small absgeom, \small relerr, \small primDR, \small primDR\_relx\_in};
  \end{axis}
\end{tikzpicture} } \hspace{0.5cm}
    \subfloat[][Total inner iterations]{
         \begin{tikzpicture}[scale=.6]
    \begin{axis}[
    width  = 13cm,
    height = 8cm,
    major x tick style = transparent,
    ybar=0.5mm,
    bar width=2mm,
    ymajorgrids = true,
    ylabel={Cumulative inner iterations},
    symbolic x coords={Colon, Leukemia, Prostate, Arcene},
    xtick = data,
    x tick label style = {rotate=90},
    scaled y ticks= base 10:-3,
    ymin=0,ymax=30e3,
    ytick style={draw=none},
    legend cell align=left,
    legend style={
        at={(0.68,0.57)},
        anchor= south west,
        column sep=1ex
    },
    ]

    \addplot [black, fill=black!90] table[x=N,y=O] {\dispdata};
    \addplot [yellow, fill=yellow!90] table[x=N,y=P] {\dispdata};
    \addplot [blue, fill=blue!90] table[x=N,y=Q] {\dispdata};
    \addplot [green, fill=green!90] table[x=N,y=R]    {\dispdata};
    \addplot [red, fill=red!90] table[x=N,y=S]  {\dispdata};
    \legend{\small absgeom,\small relerr, \small primDR,\small FISTA , \small primDR\_relx\_in};
    \end{axis}
\end{tikzpicture}}
    \end{minipage}\par\medskip
    
    \begin{minipage}{0.98 \linewidth}
    \centering
   \subfloat[][Runtimes in seconds]{
         \begin{tikzpicture}[scale=0.6]
    \begin{axis}[
    width  = 13cm,
    height = 8cm,
    major x tick style = transparent,
    ybar=0.5mm,
    bar width=2mm,
    ymajorgrids = true,
    ylabel={Runtimes in seconds},
    symbolic x coords={Colon, Leukemia, Prostate, Arcene},
    xtick = data,
    x tick label style = {rotate=90},
    scaled y ticks= base 10:-2,
    ymin=0,ymax=9e2,
    ytick style={draw=none},
    legend cell align=left,
    legend style={
        at={(0.02,0.57)},
        anchor= south west,
        column sep=1ex },
    ]

    \addplot [black, fill=black!90] table[x=T,y=U] {\dispdata};
    \addplot [yellow, fill=yellow!90] table[x=T,y=V] {\dispdata};
    \addplot [blue, fill=blue!90] table[x=T,y=W] {\dispdata};
    \addplot [green, fill=green!90] table[x=T,y=X]    {\dispdata};
    \addplot [red, fill=red!90] table[x=T,y=Y]  {\dispdata};
    \legend{\small absgeom,\small relerr, \small primDR, \small FISTA,  \small primDR\_relx\_in};
    \end{axis}
   \end{tikzpicture}}
    \end{minipage}
    \end{figure}

\appendix
\section{Auxiliary results}
\label{sec:apx}

\begin{lemma}[See for example Proposition 20.33 of~\cite{bau.com-book}]
\label{lm:simons}
If $T$ is maximal monotone on $\R^n$, $\{(\tilde z^j,v^j)\}$ is such that
$v^j\in T(\tilde z^j)$ for all $j\geq 0$, $\lim_{j\to \infty}\,\tilde
z^j=z^\infty$, and $\lim_{j\to \infty}\,v^j=v^\infty$, then $v^\infty\in
T(z^\infty)$.
\end{lemma}


\mgap

\begin{lemma}
 \label{lm:inverse}
 The inverse function of the scalar map
\begin{align*}
 (0,2)\ni \rho\mapsto \phi(\rho):=\dfrac{2(2-\rho)}{4-\rho+\sqrt{16\rho-7\rho^2}}\in 
 (0,1)
\end{align*}
is 
\begin{align*}
 (0,1)\ni \beta \mapsto \psi(\beta)
 :=\dfrac{2(\beta-1)^2}{2(\beta-1)^2+3\beta-1}\in (0,2).
\end{align*}
\end{lemma}
\begin{proof}
We first claim that $\psi(\beta) \in [0,2]$ for all
$\beta\in[0,1]$ and $\psi(\beta) \in (0,2)$ for all $\beta\in(0,1)$.  To
establish this claim, we first note that by elementary calculus and some
simplifications, we have
\begin{equation} \label{psiprime}
\tfrac{d}{d\beta} \psi(\beta)
=
\frac{6\beta^2 - 4\beta -2}{\big(2(\beta-1)^2+3\beta-1\big)^2}
=
\frac{6\beta^2 - 4\beta -2}{\big(2\beta^2 - \beta + 1\big)^2}.
\end{equation}
The discriminant of $2\beta^2 - \beta + 1$ is negative, so it has no real
roots and the denominator of~\eqref{psiprime} is always positive.  The
expression in the numerator is convex and applying the quadratic formula
yields that that its roots are $-1/3$ and $1$, so therefore it is nonpositive
on $[0,1]$ and negative on $(0,1)$. Therefore, $\tfrac{d}{d\beta} \psi(\beta)$
exists for all  $\beta \in [0,1]$ and is negative for all $\beta\in(0,1)$,
implying that $\psi$ is a decreasing function on $(0,1)$. By direct
calculation, $\psi(0)=2$ and $\psi(1) = 0$, so therefore $\big\{\psi(\beta) \;
| \; \beta\in[0,1]\big\} = [0,2]$ and $\big\{\psi(\beta) \; | \;
\beta\in(0,1)]\big\} = (0,2)$, establishing the initial claim.
To continue the proof, we next establish that
\begin{align}
  \label{eq:noite02}
 \phi(\psi(\beta))=\beta\qquad \forall \beta\in (0,1).
\end{align}
To this end, fix any $\beta\in (0,1)$ and define
\[
 (0,2)\ni \rho:= \psi(\beta)=\dfrac{2(\beta-1)^2}{2(\beta-1)^2+3\beta-1}
=\dfrac{2\beta^2-4\beta+2}{2\beta^2-\beta+1},
\]
which implies the quadratic equation
\begin{align}
 \label{eq:root_bb}
 2(1-\rho)\beta^2-(4-\rho)\beta+(2-\rho)=0.
\end{align}
We now consider three cases in~\eqref{eq:root_bb}: $\rho=1$,
$\rho<1$, and  $\rho>1$.
\begin{description}
\item[$\rho=1$:]  in this case, simplification
of \eqref{eq:root_bb} and the definition 
of $\phi$ yield that $\beta=1/3=\phi(1)$.
\item[$\rho<1$:] the unique minimizer
of the quadratic function in \eqref{eq:root_bb} is
$\beta^*:=(4-\rho)/\big(4(1-\rho)\big)$, which must be greater than $1$
because $\rho > 0$. Thus, we have $\beta^*>1>\beta>0$, so
$\beta$ is the smaller root of the quadratic equation in
\eqref{eq:root_bb}.  Using the quadratic formula and rationalizing the denominator,
\begin{align}
 \beta &= 
\frac{4-\rho - \sqrt{(\rho-4)^2 - 4 \cdot 2(1-\rho)(2-\rho)}}{2\cdot 2(1-\rho)}
= \frac{4-\rho - \sqrt{16\rho - 7\rho^2}}{4(1-\rho)} 
\label{minusroot} \\
&= \frac{4-\rho - \sqrt{16\rho - 7\rho^2}}{4(1-\rho)} 
  \cdot \frac{4-\rho + \sqrt{16\rho - 7\rho^2}}{4-\rho + \sqrt{16\rho - 7\rho^2}} 
  \nonumber \\
&= \frac{16 - 24\rho + 8 \rho^2}{4(1-\rho)\big(4-\rho+\sqrt{16\rho-7\rho^2}\big)}
= \frac{8(1-\rho)(2-\rho)}{4(1-\rho)\big(4-\rho+\sqrt{16\rho-7\rho^2}\big)} 
\nonumber \\
&= 
 \dfrac{2(2-\rho)}{4-\rho+\sqrt{16\rho-7\rho^2}} = \phi(\rho).  \label{eq:noite}
\end{align}
\item[$\rho > 1$:] in this case, $\beta^*$ as defined in the previous case is
the unique maximizer of the quadratic function in
\eqref{eq:root_bb} and $\beta^* < 0$. So $\beta^*<0<\beta<1$ and $\beta$ is
the larger root of the quadratic in~\eqref{eq:root_bb}.  Since the coefficient
of the quadratic term is negative in this case, this root also takes the
form~\eqref{minusroot}, and consequently~\eqref{eq:noite} still holds.
\end{description}

The proof of \eqref{eq:noite02} is now complete. Finally, we now prove that
\begin{align}
  \label{eq:noite03}
 \psi(\phi(\rho))=\rho\qquad \forall \rho\in (0,2).
\end{align}
To this end, let $0<\rho<2$ and define
\[
 (0,1)\ni \beta:= \phi(\rho)=\dfrac{2(2-\rho)}{4-\rho+\sqrt{16\rho-7\rho^2}}.
\]
Using the above definition and the quadratic formula, we
conclude that $\beta$ also satisfies the quadratic equation
\eqref{eq:root_bb}, which after some simple algebra gives
\[
\rho=\dfrac{2(\beta-1)^2}{2(\beta-1)^2+3\beta-1},
\]
that is, $\rho=\psi(\beta)$, which in turn
is equivalent to
\eqref{eq:noite03}. 
\end{proof}

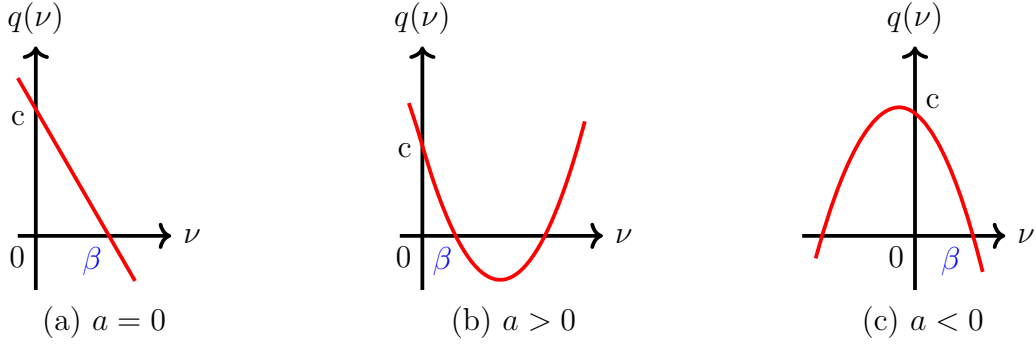
\begin{figure}
\centering
\begin{minipage}{0.3\textwidth}
\centering
\begin{tikzpicture}[domain=-0.5:2, scale=0.6]
    \draw[->,line width = 0.50mm] (-0.4,0) -- (3,0) node[right] {$\nu$};
    \draw[->,line width = 0.50mm] (0,-1.2) -- (0,4.2) node[above] {$q(\nu)$};
    \draw[color=red,line width = 0.50mm] (-0.4,3.5) -- (2.2,-1);
        \node[below left,black] at (0,0) {0};
        \node[below left,black] at (0,3) {c};
        \node[below left,blue] at (1.7,0) {$\beta$};
\end{tikzpicture} \\
(a) $a=0$ 
\end{minipage}
~
\begin{minipage}{0.3\textwidth}
\centering
\begin{tikzpicture}[domain=-0.3:3.6, scale=0.6]
    \draw[->,line width = 0.50mm] (-0.5,0) -- (4,0) node[right] {$\nu$};
    \draw[->,line width = 0.50mm] (0,-1.2) -- (0,4.2) node[above] {$q(\nu)$};
    \draw[red, line width = 0.50mm]   plot[smooth,domain=-0.3:3.6] (\x, {\x^2 -3.45*(\x)+2});
    \node[below left,black] at (0,0) {0};
    \node[below left,black] at (0,2.3) {c};
    \node[below left,blue] at (0.9,0) {$\beta$};
    \end{tikzpicture} \\
(b) $a>0$
\end{minipage}
~
\begin{minipage}{0.3\textwidth}
\centering
  \begin{tikzpicture}[domain=-0.5:4, scale=0.6]
    \draw[->,line width = 0.50mm] (-2.5,0) -- (2,0) node[right] {$\nu$};
    \draw[->,line width = 0.50mm] (0,-1.2) -- (0,4.2) node[above] {$q(\nu)$};
    \draw[color=red,line width = 0.50mm] (-2.2,-0.5) parabola[parabola height= 3.5cm] + (3.7,-0.3);
    \node[below left,black] at (0,0) {0};
    \node[above right,black] at (0,2.6) {c};
    \node[below left,blue] at (1.25,0) {$\beta$};
    \end{tikzpicture} \\
(c) $a<0$
\end{minipage}
  \caption{Possible cases for the quadratic function $q(\cdot)$ 
           in Lemma \ref{lm:quadratic}.}
  \label{fig05}
\end{figure}

\begin{lemma}
 \label{lm:quadratic}
 Let $\R\ni \nu\mapsto q(\nu):=a \nu^2-b\nu+c$ be a real function and assume that $b,c>0$ and $b^2-4ac>0$. Define
\begin{align}
  \label{eq:root_b}
 \beta:= \dfrac{2c}{b+\sqrt{b^2-4ac}}>0.
\end{align}
\begin{itemize}
 \item[\emph{(i)}] If $a=0$, then $q(\cdot)$ is a decreasing affine function
 and $\beta>0$ as in \eqref{eq:root_b} is its unique root \emph{(}see
 \emph{Figure \ref{fig05}(a)}\emph{)}.
\item[\emph{(ii)}] If $a>0$ \emph{(}resp. $a<0$\emph{)}, then $q(\cdot)$ is a
 convex \emph{(}resp. concave\emph{)} quadratic function and $\beta>0$ as in
 \eqref{eq:root_b} is its smallest \emph{(}resp. largest\emph{)} root
 \emph{(}see \emph{Figure \ref{fig05}(b)} and \emph{Figure \ref{fig05}(c)},
 resp.\emph{)}.
\end{itemize}
In both cases $\emph{(i)}$ and $\emph{(ii)}$,  $\beta>0$ as in
\eqref{eq:root_b} is a root of $q(\cdot)$, and $q(\cdot)$ is decreasing in the
interval $[0,\beta]$ \emph{(}see \emph{Figure \ref{fig05}}\emph{)}.
\end{lemma}
\begin{proof}
 The proof of (i) is straightforward. To prove (ii), note that rationalizing
the denominator of~\eqref{eq:root_b} results in
$\beta=\left(b-\sqrt{b^2-4ac}\right)/2a$, which in turn implies that (ii)
follows from the quadratic formula and the assumption that $b,c>0$. The last
statement of the lemma is a direct consequence of (i), (ii) and the assumption
that $b,c>0$.
\end{proof}

\begin{lemma}[Opial~\cite{opial-wea.bams67}]
 \label{lm:opial}
Let $\emptyset \neq \Omega\subset \R^n$ and $\{z^k\}$ be a sequence in $\R^n$
such that every cluster point of $\{z^k\}$ belongs to $\Omega$ and $\lim_{k\to
\infty}\,\norm{z^k-z^*}$ exists for every $z^*\in \Omega$. Then $\{z^k\}$
converges to a point in $\Omega$.
\end{lemma}

\noindent The following lemma was essentially proved by Alvarez and Attouch in
\cite[Theorem 2.1]{alv.att-iner.svva01}.

\begin{lemma}
 \label{lm:alv.att}
Let the sequences $\{\varphi_k\}$, $\{s_k\}$, $\{\alpha_k\}$ and $\{\delta_k\}$ in $[0,+\infty[$
and $\alpha\in \R$ be such that
$\varphi_0=\varphi_{-1}$, $0\leq \alpha_{k}\leq \alpha<1$ and
\begin{align}
  \label{eq:alv.att02}
\varphi_{k+1}-\varphi_{k}+s_{k+1}\leq \alpha_{k}(\varphi_{k}-\varphi_{k-1})+\delta_{k}\qquad \forall k\geq 0.
\end{align}
The following hold:
\begin{enumerate}
  \item [\emph{(a)}] For all $k\geq 1$,
 \begin{align}
       \label{eq:alv.att01}
       \varphi_k+\sum_{j=1}^k\,s_j\leq
      \varphi_0+\dfrac{1}{1-\alpha} \sum_{j=0}^{k-1}\,\delta_j.
    \end{align}
  \item [\emph{(b)}] If $\sum^{\infty}_{k=0}\delta_k <+\infty$, then
  $\lim_{k\to \infty}\,\varphi_{k}$ exists, \emph{i.e.}, the
  sequence $\{\varphi_k\}$ converges to some element in $[0,\infty)$.
\end{enumerate}
\end{lemma}
\begin{proof}
 It was proved in \cite[Theorem 2.1]{alv.att-iner.svva01} that
 $\mathcal{M}:=(1-\alpha)^{-1}\sum_{j=0}^{k}\delta_j \geq
\sum_{j=1}^{k+1}\,[\varphi_j-\varphi_{j-1}]_+$, where $[\cdot]_{+}=\max \{\cdot,0\}.$
Using this, the  assumptions $\varphi_0=\varphi_{-1}$, $0\leq \alpha_{k}\leq \alpha<1$ and \eqref{eq:alv.att02}, and some algebraic manipulations we find, for all $k\geq 0$,
\begin{align*}
 \varphi_{k+1}+\sum_{j=1}^{k+1}\,s_j &\leq \varphi_0+ \alpha\sum_{j=1}^{k+1}[\varphi_j-\varphi_{j-1}]_{+}+\sum_{j=0}^{k}\,\delta_j\\
      &\leq \varphi_0+ \alpha \mathcal{M}+(1-\alpha)\mathcal{M}=\varphi_0+\mathcal{M},
    \end{align*}
which proves (a). To finish the proof, we note that (b)
was established within the proof of \cite[Theorem
2.1]{alv.att-iner.svva01}.
\end{proof}


%
\def\cprime{$'$}

\end{document}